\documentclass[11pt,a4paper]{article}
\usepackage[utf8]{inputenc}
\usepackage{a4wide} %compact layout by simone
\usepackage[utf8]{inputenc}
\usepackage{subfig}
\usepackage{bbm}
\usepackage{amsmath}
\usepackage{graphicx}
\usepackage{epstopdf}
\usepackage{algcompatible}
\usepackage{adjustbox}
\usepackage{titlesec}
\usepackage{csquotes}
\usepackage{tikz}
\usetikzlibrary{shapes.geometric}
\usepackage{makecell}
\usetikzlibrary{positioning,arrows.meta}
\RequirePackage{ifthen}
\RequirePackage[T1]{fontenc}
\RequirePackage[ngerman,english]{babel}
\RequirePackage{amsmath}
\RequirePackage{amssymb}
\RequirePackage{amsthm}
\RequirePackage{amsfonts}
\RequirePackage{graphicx}
\RequirePackage{color}
\RequirePackage{listings} 
\usepackage{epstopdf}

\usepackage{algcompatible}
\RequirePackage[section,Algorithm]{algorithm}
\RequirePackage{emptypage}

\usepackage[backend=biber,sorting=nyt]{biblatex}
\renewbibmacro{in:}{}
\DeclareFieldFormat[article]{title}{\textit{#1}} 
\DeclareFieldFormat[article]{journaltitle}{#1}
\theoremstyle{plain}						% Default
\newtheorem{theorem}{Theorem}[section]
\newtheorem{lemma}[theorem]{Lemma}

\newtheorem{definition}[theorem]{Definition}
\newtheorem{example}[theorem]{Example}
\newtheorem{remark}[theorem]{Remark}
\numberwithin{equation}{section}

\titleformat{\section}[block]{\normalfont\bfseries}{\thesection.}{0.5em}{}
\titlespacing{\section}{0pc}{1pc}{1pc}

\titleformat{\subsection}[block]{\normalfont\bfseries}{\thesubsection.}{0.5em}{}
\titlespacing{\subsection}{0pc}{1pc}{1pc}

\begin{document}
	
	\title{A Hyperbolic Transport Model for Passenger Flow on Tram Networks}
	
	\author{Thomas Schillinger\footnotemark[1]}
	
	\footnotetext[1]{University of Mannheim, Department of Mathematics, 68131 Mannheim, Germany (schillinger@uni-mannheim.de)}

	\date{ \today }
	\maketitle
	
	\begin{abstract}
		\noindent
		We introduce a modeling framework for an urban tram network based on a hyperbolic partial differential equation describing the transport of passengers along the network, coupled with a family of stochastic processes representing passenger boarding. Solutions are considered in a measure-valued sense. The system is further extended and subjected to uncertainties such as delays and service interruptions through a numerical study. Its robustness is assessed using appropriate risk measures.
	\end{abstract}
	
	% Conservation laws, Traffic problems, Finite difference methods
	
	measure-valued solutions
	
	%%%%%%%%%%%%%%%%%%%%%%%%%%%%%%%%%%%%%%%%%%%%
	
	\section{Introduction}	
	Modeling and analyzing passenger dynamics in urban mobility systems, such as tram or light rail networks, has become increasingly important due to rising demands on public transportation infrastructure, efficiency, and robustness. Urban mobility systems are characterized by strong spatial–temporal variability, stochastic demand patterns, and frequent operational disturbances. Mathematical modeling provides a systematic framework to understand the fundamental mechanisms governing passenger transport, congestion formation, and system bottlenecks, particularly under uncertainty and real-time variability.
	
	Hyperbolic partial differential equations are a well-established modeling paradigm for transport phenomena in mobility systems. The most prominent example is vehicular traffic flow, initiated by the seminal works of Lighthill, Whitham, and Richards \cite{Lighthill.1955,RICHARDS.1956} and further developed in numerous extensions and network formulations; see \cite{GARAVELLO.2016,GARAVELLO.2006} for comprehensive overviews. Related approaches have been proposed for pedestrian dynamics \cite{Colombo.2012}, and more recently, multi-modal traffic models incorporating additional mobility forms such as bicycle lanes have been studied \cite{Joumaa.2026}.
	Despite this extensive literature, the mathematical modeling of public transportation systems - particularly tram or light rail networks - has received comparatively less attention from a PDE-based perspective.
	
	Existing approaches to tram and railway systems are often formulated on the level of vehicle trajectories, timetables, or event-based models. In this context, ordinary differential equations typically appear implicitly, for instance as part of optimization or scheduling frameworks rather than as standalone dynamical models. A prominent line of research concerns the optimization of cyclic timetables via Periodic Event Scheduling Problems (PESP), introduced in \cite{Serafini1989} and actively studied in recent works, see e.g.\ \cite{bortoletto2024,Masing2023}. Robustness aspects of tram systems have been studied, for instance, for the Cologne tram network in \cite{Lueckerath2012} and under disruption scenarios in \cite{Roelofsen2018}. For a broader survey including physical and mechanical aspects of railway systems, we refer to \cite{Iwnicki.2019}. While these approaches are highly relevant for operational planning, they are not primarily designed to describe passenger transport dynamics along the network. To the best of our knowledge, a systematic PDE-based framework focusing on passenger transport in tram networks remains largely unexplored.
	
	A natural starting point for modeling spatial transport phenomena is the linear transport equation. In its simplest form, it describes the evolution of a transported quantity $\rho(x,t)$ in space and time. In the context of tram-based mobility systems, $\rho(x,t)$ may represent the number or density of onboard passengers at position $x$ and time $t$. Passenger boarding and alighting processes occur at discrete locations, namely at tram stops, and must therefore be incorporated through boundary conditions or localized source terms at the network nodes. The linear transport equation has been extensively studied in mathematical theory. Under suitable assumptions on the transport velocity, initial data, and source terms, the method of characteristics yields explicit solution representations \cite{LEVEQUE.2002}. For nonsmooth coefficients, existence and uniqueness of weak or distributional solutions have been analyzed, for instance, in \cite{DiPerna1989,Evers2015,Gugat.2015f}. Transport equations with uncertain velocities are considered in \cite{Dorini.2011}, optimal control problems based on linear transport equations are studied in \cite{LiRao2004}, and uncertainty in boundary data is addressed in \cite{Gottlich.2019b,Gottlich.2022b}. An application of a transport equation with a source term to vehicular traffic flow is presented in \cite{Herty2003}, where lane changes are modeled within a multi-lane traffic framework. These references represent only a small selection of the extensive literature on transport equations.
	
	Measure-valued solution concepts play a central role when classical or weak formulations are no longer adequate. For the linear transport equation, measure-valued solutions have been investigated in \cite{BERTSCH.2024,Evers2015}, and extensions to solution-dependent velocities are discussed in \cite{Evers2018}. Models with additional solution-dependent source terms are analyzed in \cite{Piccoli2014}. An important extension of transport equations to network structures was introduced in \cite{Camili2017} and further developed to include nonlocal flux functions in \cite{Camilli2018}. Measure-valued solutions for nonlinear transport equations have also been studied, see \cite{Gwiazda2012}.
	
	From a modeling perspective, measure-valued formulations provide a flexible framework for incorporating real-world features into transport models. In particular, passenger boarding and alighting at tram stops can be represented by Dirac measures in space and time, enforcing that passenger exchange occurs only at designated network nodes. Such singular boundary conditions typically fall outside the scope of standard weak formulations and require a more general notion of solution. This makes measure-valued approaches particularly well-suited for passenger transport modeling on public transportation networks.
	
	The present work investigates a hyperbolic partial differential equation - specifically, the linear transport equation - posed on a network. In Section \ref{sec: modellingTramNetwork} we equip it with additional coupling and admissibility conditions at the vertices, motivated by tram-based mobility systems. %Our aim is to bridge rigorous mathematical theory and practical transportation modeling. 
	Beyond analytical existence results for measure-valued solutions, we incorporate stochastic passenger inflow processes in Section \ref{sec: PDE}, inspired by the works \cite{Camili2017,Evers2015}. The impact of operational uncertainties, such as delays or partial service interruptions, are studied through numerical simulations. These are used to assess system robustness using suitable performance and risk measures related to passenger congestion in Section \ref{sec:numerik}.

	\section{Modeling a tram network}
	\label{sec: modellingTramNetwork}
	We consider a directed graph $\mathcal{G} = (\mathcal{V},\mathcal{E})$, where $\mathcal{V}$ are the vertices and $\mathcal{E}$ are the directed edges connecting vertices. Each vertex $v \in \mathcal{V}$ corresponds to one tram stop. The length of an edge $e \in \mathcal{E}$ is denoted by $l_e>0$ and describes the distance between two consecutive stops. %We define be $p(e)$ the set of predecessor edges of $e$, as long as $e$ is not the first edge in the network. 
	In addition $v^- \subset \mathcal{E}$ and $v^+\subset \mathcal{E}$ define the set of all ingoing and outgoing edges of $v \in \mathcal{V}$. We assume that edges are directed. In practice this maintains the idea of two individual tram tracks connecting two vertices, where the trams for the direction $v_1$ to $v_2$ all run on one track (edge), whereas the trams going from $v_2$ to $v_1$ use another track. 
	Each edge $e$ is represented by the interval $[0,l_e]$. 
	
	At each vertex $v \in \mathcal{V}$ passengers can board and alight the tram. %{Instead of considering the boarding and alighting of passengers at $v$, it will be useful to shift the getting on and off of the passengers at $x=0$ of each edge $e \in \mathcal{E}$.} 
	At special tram stations $\mathcal{V}^\text{start}\subset \mathcal{V}$ trams start their journeys. All source nodes of the graph belong to $\mathcal{V}^\text{start}$, but it is also possible that an inner vertex is a starting point for tram journeys. Similarly we define the set $\mathcal{V}^\text{terminal}\subset \mathcal{V}$ as the tram stations where trams terminate. Trams can move on the network with some given velocity. %The velocity is described by a family of piecewise Lipschitz, bounded and strictly positive functions $(w^e(x))_{e\in \mathcal{E}}: [0,l_e] \rightarrow \mathbb{R}_{>0}$. 
	The velocity is described by a family of strictly positive constants $(w^e)_{e\in \mathcal{E}}$. %\in \mathbb{R}_{>0}$. 
	Thus the trams move with a strictly positive but bounded velocity.
	
	\subsection{Admissible tram schedule}
	\label{sec: admissibleTramSchedule}
	In reality trams run according to an externally given tram schedule. To transfer this idea to the tram model we define a family of capacity functions $(\tau^e(t))_{e \in \mathcal{E}}$. These encode both the spatial-temporal movement of trams and the tram capacity at the origin vertex of each edge $e$ at time $t$. Specifically, $\tau^e(t) = 0$ if no tram departs along edge $e$ at time $t$, otherwise it equals the tram's capacity. Equivalently, this function can be interpreted as a \textbf{tram schedule}, providing a formal description of the network’s operational constraints.  
	Before proceeding, we introduce a weaker notion of injectivity that will be required for the subsequent analysis.
	\begin{definition}
		\label{def: injectiveEmptySet}
		Let $A,B$ be two finite sets. A function $f: A \mapsto \left( B \bigcup \emptyset \right)$ is called \textbf{injective except for the empty set} if
		\begin{align*}
			\forall x,y \in A: \quad f(x)=f(y) \land f(x) \ne \emptyset \Rightarrow x=y.
		\end{align*}
	\end{definition}
	\begin{definition}
		\label{def: admissibleSchedule}
		A family of functions $\tau = (\tau^e)_{e \in \mathcal{E}}$ where $\tau^e: [0,\infty) \rightarrow \mathbb{R}_{\geq 0}$ is called an \textbf{admissible tram schedule} if
		\begin{itemize}
			\item[a)] For any $e \in \mathcal{E}$, $\tau^e(t)<\infty$ for all $t \in [0,\infty)$.
				\item[b)] For any $t\in[0,\infty)$ and $v\in\mathcal{V}\backslash\mathcal{V}^\text{terminal}$, there exists a time parameterized \textit{injective except for the empty set} map $\delta_t^v: v^- \rightarrow \lbrace v^+ \bigcup \emptyset\rbrace$ such that for any $e \in v^-$ with $\tau^e\left(t-\frac{l_e}{w^e}\right) >0$ the function $\delta_t^v$ maps to exactly one of the outgoing edges from the set $v^+$. Furthermore it holds 
			\begin{align*}
				\tau^{\delta_t^v(e)}(t) = \tau^e\left(t-\frac{l_e}{w^e}\right).
			\end{align*}
			For any $e \in v^-$ with $\tau^e\left(t-\frac{l_e}{w^e}\right)=0$, we set $\delta_t^v(e) = \emptyset$. If $v \in \mathcal{V}^\text{terminal}$ then for any $e \in v^-$ with $\tau^e\left(t-\frac{l_e}{w^e}\right)>0$ the function $\delta_t^v$ may also map to $\emptyset$ (terminating tram; particularly $v^+$ may be empty). 
			\item[c)] The inverse of $(\delta_t^v)^{-1}: v^+ \rightarrow \lbrace v^- \bigcup \emptyset\rbrace$ is defined for any $v \in \mathcal{V} \backslash \mathcal{V}^\text{start}$ and $t\in[0,\infty)$ with $e^+ \in v^+$ by
			\begin{align*}
				(\delta_t^v)^{-1}(e^+) = \begin{cases}
					e^-, & \text{if there exists } e^-\in v^- \text{  with  } \delta_t^v(e^-) = e^+,\\
					\emptyset, & \text{else.}
				\end{cases}
			\end{align*}
			By injectivity except for the empty set according to Definition \ref{def: injectiveEmptySet} of $\delta_t^v$ the edge $e^- \in v^-$ is uniquely defined in the first case.
			%$e \in v^+$ for which $p(e)\neq \emptyset $, $\mu^e (0,t)>0$, we have that there exists (exactly one??) $\tilde{e} \in v^-$ such that
			%  \begin{align*}
				%      \mu^e(0,t) = \mu^{\tilde{e}}(l_{\tilde{e}} ,t)
				%  \end{align*}
			%  and furthermore the sum of all ingoing capacities should match the outgoing capacities
			%  \begin{align*}
				%      \sum_{e\in v^-} \mu^e(l_e, t) = \sum_{e\in v^+} \mu^e(l_e, t).
				%  \end{align*}
			% \item mit dem zweiten Teil der Definition kann innerhalb des systems  keine Straßenbahn starten
		\end{itemize}
		
	\end{definition}
		The time $t-\frac{l_e}{w^e}$ describes when the tram entered the previous edge $e$. %and is calculated by the fraction of length of $e$ and the velocity of the tram on %$e$.
		The map $\delta_t^v$ assigns to each edge from which a tram approaches a vertex $v \in \mathcal{V}$ the corresponding outgoing edge. If no tram approaches the vertex from a given edge, that edge is mapped to the empty set. This definition allows for more than one tram to arrive at a tram stop (which may occur due to multiple tracks at the stop), while it excludes the possibility that multiple trams leave the tram stop simultaneously on the same outgoing edge. This is ensured by the injectivity assumption, except for the empty set.
		
		Furthermore, this definition enforces the following properties: the incoming capacity from a single edge cannot split at the vertex but must be transferred entirely to exactly one outgoing edge. Consequently, trams consisting of multiple cars cannot be separated within the model. Likewise, the merging of different trams into a single tram is not permitted.

	\paragraph{Conditions at the vertex for $m-n$ junctions}
	\label{sec: junctionTypes}
	To illustrate how the admissible tram schedule affects the dynamics at vertices, we begin with the simplest 1-1 junctions and progressively consider more complex scenarios, allowing multiple incoming and outgoing edges. As prototypes of such networks, we discuss 1-2 and 2-1 junctions, before generalizing to arbitrary $m$-$n$ junctions.
	
	\paragraph{1-1 junction.}
	
	The 1-1 junction represents the simplest case. A sequence of 1-1 junctions models a tram network with a single isolated line, where every vertex corresponds to a tram stop. Here, $v^-$ and $v^+$ each contain exactly one edge, representing the unique predecessor and successor edge (except at $\mathcal{V}^\text{start}$ and $\mathcal{V}^\text{terminal}$). Denote these edges by $e^- \in v^-$ and $e^+ \in v^+$. Then, the function $\delta_t^v$ from Definition~\ref{def: admissibleSchedule} is defined by
	\begin{align*}
		\delta_t^v(e^-) = \begin{cases}
			e^+, & \tau^{e^+}(t) > 0,\\
			\emptyset, & \tau^{e^+}(t) = 0.
		\end{cases}
	\end{align*}
	
	\paragraph{1-2 junction.}
	
	In a 1-2 junction, a tram arriving from a single direction may proceed along one of two outgoing edges. For vertex $v \in \mathcal{V}$ and $e^- \in v^-$, let $e_1^+, e_2^+ \in v^+$ be the outgoing edges. Then
	\begin{align*}
		\delta_t^v(e^-) = \begin{cases}
			e_1^+, & \tau^{e_1^+}(t) > 0,\\
			e_2^+, & \tau^{e_2^+}(t) > 0,\\
			\emptyset, & \tau^{e_1^+}(t) = \tau^{e_2^+}(t) = 0.
		\end{cases}
	\end{align*}
	By construction, both $\tau^{e_1^+}(t)$ and $\tau^{e_2^+}(t)$ cannot be simultaneously positive, as trams do not split.
	
	\paragraph{2-1 junction.}
	
	A 2-1 junction has two incoming edges $e_1^-, e_2^- \in v^-$ and a single outgoing edge $e^+ \in v^+$. The admissibility condition (Definition~\ref{def: admissibleSchedule}) prevents two trams from simultaneously entering the junction and attempting to exit along the single outgoing edge. Formally, for $i = 1,2$,
	\begin{align*}
		\delta_t^v(e_i^-) = 
		\begin{cases}
			e^+, & \tau^{e_i^-}\Bigl(t-\frac{l_{e_i^-}}{w^{e_i^-}}\Bigr) > 0,\\
			\emptyset, & \text{otherwise.}
		\end{cases}
	\end{align*}
	By construction of an admissible schedule, it is impossible that both incoming capacities are simultaneously positive at the same time, ensuring that at most one tram can exit along $e^+$.
	
	\paragraph{m-n junction.}
	
	For general $m$-$n$ junctions, an explicit form of $\delta_t^v$ depends on the trajectories of the tram lines and their interactions at the junction. The function must be defined in a way that preserves the injectivity and admissibility conditions of the schedule. We provide Example \ref{ex:2-2junction} in the case of a 2-2 junction to illustrate this. 
	
	%The generalization of the previous sections are $m-n$-junctions for, $m,n \in \mathbb{N}$. TODOTODOTODO
	%The ideas from the three previous subcases can be applied here. todo
	%vielleicht auch ganz spannend (und dann am Beispiel Alte Feuerwache das aufziehen?)
	
	%Hier die Bedingungen aufschreiben, die wir an groeßeren Kreuzungen brauchen

	\begin{example}
		\label{ex:2-2junction}
		We consider a 2-2 junction at a vertex $v \in \mathcal{V}$ with $e_1^-, e_2^- \in v^-$ and $e_1^+,e_2^+ \in v^+$. Assume that there are four tram lines running at $v$ where each line is running every 10 minutes:
		\begin{itemize}
			\item Line 1: taking the route over $e_1^-$ and $e_1^+$, always at minute 1, 11, 21,...
			\item Line 2: taking the route over $e_1^-$ and $e_2^+$, always at minute 4, 14, 24,...
			\item Line 3: taking the route over $e_2^-$ and $e_1^+$, always at minute 6, 16, 26, ...
			\item Line 4: taking the route over $e_2^-$ and $e_1^+$, always at minute 4, 14, 24, ...
		\end{itemize}
		The lines at $v$ are graphically shown in Figure \ref{fig: Lines_example4}.
		
		\begin{figure}[ht!]
			\centering
			
			\begin{tikzpicture}[line width=2pt, scale=1, every node/.style={scale=1}]
				
				% Farben definieren
				\definecolor{tramred}{RGB}{220,30,30}
				\definecolor{tramgreen}{RGB}{0,150,0}
				\definecolor{tramblue}{RGB}{30,60,200}
				\definecolor{tramorange}{RGB}{255,130,0}
				
				% Haltestelle v
				\node[circle,draw=black,fill=white,minimum size=10mm] (v) at (0,0) {\phantom{\Large v}};
				
				% Kanten-Beschriftungen
				\node at (-3.25,1.2) {$e_1^-$};
				\node at (-3.25,-1.2) {$e_2^-$};
				\node at (3.1,1.4) {$e_1^+$};
				\node at (3.1,-1.3) {$e_2^+$};
				
				% Linie 1: e1- -> v -> e1+ (leicht oberhalb)
				\draw[tramred] (-3,1.15) -- (0,0.15) -- (3,1.15) node[midway, above, sloped] {\footnotesize Linie 1 (min 1)};
				
				% Linie 2: e1- -> v -> e2+
				\draw[tramgreen] (-3,1) -- (0,0) -- (3,-1) node[midway, above right, sloped] {\footnotesize Linie 2 (min 4)};
				
				% Linie 3: e2- -> v -> e1+
				\draw[tramblue] (-3,-1) -- (0,0) -- (3,1) 
				node[midway, below right, sloped,xshift = -6pt, yshift=-5pt]  {\footnotesize Linie 3 (min 6)};
				
				% Linie 4: e2- -> v -> e1+ (leicht unterhalb von Linie 3)
				\draw[tramorange] (-3,-1.15) -- (0,-0.15) -- (3,0.85) 
				node[pos=-0.35, sloped, below, xshift=-15pt, yshift=-5pt] {\footnotesize Linie 4 (min 4)};
				
			\end{tikzpicture}
			\caption{Tram line routes at the junction $v$. Line 1 (red), line 2 (green), line 3 (blue), line 4 (orange).}
			\label{fig: Lines_example4}
		\end{figure}
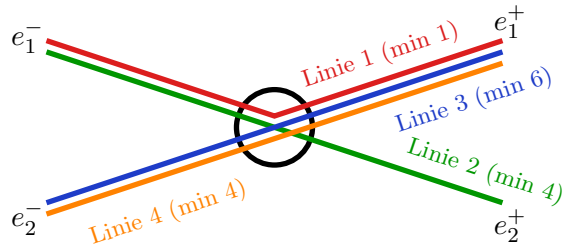
		% \begin{figure}
			%     \centering
			%     \includegraphics[width=0.5\linewidth]{netzkarte_tram_junction_adjusted.eps}
			%     \caption{Tram line routes at the junction. Line 1 (red), line 2 (green), line 3 (blue).}
			%     \label{fig: Lines_example}
			% \end{figure}
		%     \begin{figure}[ht!]
			%     \centering
			%     \includegraphics[width=0.5\linewidth]{netzkarte_tram_junction_adjusted_2.pdf}
			%     \caption{Tram line routes at the junction $v$. Line 1 (red), line 2 (green), line 3 (blue).}
			%     \label{fig: Lines_example}
			% \end{figure}
		Consider an one-hour time interval, $t\in [0,60]$. 
		Then the admissible tram schedule at the tram stops are given by the functions $\tau^{e_1^+}(t)$ and $\tau^{e_2^+}(t)$ and have the form
		\begin{align*}
			\tau^{e_1^+}(t) = \begin{cases}
				c_0^{l_1}, & t=1\\
				c_1^{l_1}, & t=11\\
				c_2^{l_1}, & t=21\\
				c_3^{l_1}, & t=31\\
				c_4^{l_1}, & t=41\\
				c_5^{l_1}, & t=51\\
				c_0^{l_3}, & t=6\\
				c_1^{l_3}, & t=16\\
				c_2^{l_3}, & t=26\\
				c_3^{l_3}, & t=36\\
				c_4^{l_3}, & t=46\\
				c_5^{l_3}, & t=56\\
				c_0^{l_4}, & t=4\\
				c_1^{l_4}, & t=14\\
				c_2^{l_4}, & t=24\\
				c_3^{l_4}, & t=34\\
				c_4^{l_4}, & t=44\\
				c_5^{l_4}, & t=54\\
				0, & \text{else}
			\end{cases}, ~~~\tau^{e_2^+}(t) = \begin{cases}
				c_0^{l_2}, & t=4\\
				c_1^{l_2}, & t=14\\
				c_2^{l_2}, & t=24\\
				c_3^{l_2}, & t=34\\
				c_4^{l_2}, & t=44\\
				c_5^{l_2}, & t=54\\
				0, & \text{else}
			\end{cases},
		\end{align*}
		with constants $0<c_j^{l_m}<\infty$ modelling the capacity of the $j$-th tram $(j=\lbrace 0,\dots,5 \rbrace)$ on line $m=1,2,3,4$.
		
		For $k = \lbrace 0,\dots,5 \rbrace$ the functions $\delta_t^v$ are defined in the following way
		\begin{align*}
			\delta_t^v(e_1^-) = \begin{cases}
				e_1^+, & t = 1 + 10k,\\
				e_2^+, & t = 4 + 10k,\\
				\emptyset, & \text{else}
			\end{cases}, ~~~~ 
			\delta_t^v(e_2^-) = \begin{cases}
				e_1^+, & t = 4 + 10k,\\
				e_1^+, & t = 6 + 10k,\\
				\emptyset, & \text{else.}
			\end{cases},
		\end{align*}
		Note that this schedule is admissible according to Definition~\ref{def: admissibleSchedule}, despite the fact that two tram lines arrive at the tram stop at time \( t = 4 \). At this time, one line runs from \( e_1^- \) to \( e_2^+ \), while the other runs from \( e_2^- \) to \( e_1^+ \). Hence, the injectivity of the map \( \delta_t^v \), except for the empty set, is not violated for \( t = 4 \). The schedule would no longer be admissible if the tram arriving from \( e_2^- \) were routed to \( e_2^+ \).
	\end{example}

	The definition of an admissible schedule now lays the foundation of the main goal of this work which is to describe the passenger behavior in a tram network.

\section{A system of hyperbolic balance laws for tram passengers}
\label{sec: PDE}

To model the evolution of passenger dynamics on the tram network, we first describe the movement of the trams themselves. Along the interior of each edge, tram motion is governed by a linear transport equation with a corresponding family of velocity functions $(w^e)_{e \in \mathcal{E}}$. Consequently, the evolution of the passenger mass along edge $e \in \mathcal{E}$ can also be described by a linear transport equation.

Let $(\rho^e)_{e \in \mathcal{E}} : [0,l_e] \times [0,\infty) \to \mathbb{R}_{\ge 0}$ denote the passenger mass on edge $e$ at position $x \in [0,l_e]$ and time $t>0$. Then
\begin{align}
	\label{eq: PDEedge}
	\rho^e_t(x,t) + \bigl(w^e\rho^e(x,t)\bigr)_x = 0,
\end{align}
%\begin{align}
%	\label{eq: PDEedge}
%	\rho^e_t(x,t) + \bigl(w^e(x)\,\rho^e(x,t)\bigr)_x = 0,
%\end{align}
where $w^e$ is a strictly positive constant for every edge. To account for passenger boarding and alighting, we introduce a source term on the right-hand side:
\begin{align}
	\label{eq: PDEedgeBalance}
	\rho^e_t(x,t) + \bigl(w^e\rho^e(x,t)\bigr)_x = g(x,t,\rho^e(x,t)).
\end{align}
Since tram stations are located at the vertices of the graph, we assume no passenger exchange occurs in the interior of an edge, i.e., $g(x,t,\rho^e) = 0$ for all $x \in (0,l_e)$. Passenger boarding at the origin vertex and alighting at the terminal vertex are modeled as Dirac measures at the endpoints:
\begin{align}
	\label{eq: PDEedgeBalanceParticular}
	\rho^e_t(x,t) + \bigl(w^e\rho^e(x,t)\bigr)_x 
	= b^e(t,\rho^e(0,t)) \,\varepsilon_0(x) - a^e(t,\rho^e(l_e,t)) \,\varepsilon_{l_e}(x),
\end{align}
where $\varepsilon_y$ denotes the Dirac measure at $y$. 

The functions $b^e$ and $a^e$ are discussed in Section~\ref{sec: boardingAlighting}. They represent the mass added to the system by boarding passengers at $x=0$ and removed by alighting passengers at $x=l_e$. Linear transport equations with source terms concentrated at points have been studied in \cite{Evers2015,Kazuaki2014}. This formulation naturally incorporates passenger exchanges at the vertices and, in principle, allows for non-permanent stops within $(0,l_e)$ without altering the network topology.
	%, which is initially — and unless otherwise specified — assumed to be constant with $w^e \equiv 1$ for any $e \in \mathcal{E}$. 

We equip each edge with an initial passenger density
\begin{align}
	\label{eq:initialDataPDEedge}
	\rho^e(x,0) = \rho_0^e(x).
\end{align}
As the spatial domain is bounded, we must also specify boundary conditions at \(x = 0\). In the standard literature, this is done by prescribing an external inflow \(f_{\mathrm{in}}^e\) such that
\begin{align}
	\label{eq:InflowPDEedge}
	w^e\rho^e(0,t) = f_{\mathrm{in}}^e(t).
\end{align}
As \(w^e > 0\), trams move forward along the edge, and therefore no boundary condition is required at \(x = l_e\), where only outflow occurs.

%At vertices representing tram stations, we impose flux conservation, taking into account the passenger boarding and alighting already incorporated in the right-hand side of Equation~\eqref{eq:PDEedgeBalanceParticular}. A detailed discussion of the network setting is provided in Section~\ref{sec: solNetwork}.

Thus far, we have introduced the linear hyperbolic balance law on each edge in Equations~\eqref{eq: PDEedgeBalanceParticular}--\eqref{eq:InflowPDEedge} and discussed the junction conditions in Section~\ref{sec: junctionTypes}. In the following subsection, we analyze the solution behavior of the family \((\rho^e)_{e\in\mathcal{E}}\), which is primarily determined by the boarding and alighting dynamics at the vertices. We first examine the behavior on a single edge, then the passenger dynamics at a tram station (which define the coupling conditions), and finally the full network solution in Section~\ref{sec: solNetwork}.

	\subsection{Measure-valued formulation}
	%For fixed $e\in \mathcal{E}$ and the corresponding vertex $v(e)$ passengers can only board or alight if $\tau^e(t)>0$. 
We have already observed that the only external influence on the passenger dynamics occurs at the spatial points \( x = 0 \) and \( x = l_e \) on each edge. 
Passengers may board or alight from the tram only at specifically designated times - namely, at the instants when a tram reaches a stop. 
To preserve the concept of a tram timetable, we recall Definition~\ref{def: admissibleSchedule}. 
The family of functions \( (\tau^e)_{e \in \mathcal{E}} \) is nonzero only at a finite number of time instants. 
These represent the times at which passengers instantaneously board or leave the tram. 
Moreover, the nonzero values of \( \tau^e \) reflect the tram’s capacity. 

It is important to note that our model does not assume that the tram physically stops at a station, in the sense that its transport velocity vanishes. 
Rather, we assume that boarding and alighting occur instantaneously at the prescribed times determined by the admissible schedule. 
%From a modeling perspective, one could also incorporate actual stopping dynamics, but this would require a substantially more involved analysis of the associated linear hyperbolic PDE.

In particular, \( \tau^e \) does not belong to the space \( C^k([0,\infty)) \) of continuous (or continuously differentiable) functions, nor to \( L^p([0,\infty)) \), the space of \( p \)-integrable functions on \( [0,\infty) \). 
Instead, the capacity functions are elements of the class of finite Borel measures. 
Consequently, in this chapter we do not describe the tram dynamics in terms of a density, that is, a mass-per-unit representation. Rather, we model the transport of measures concentrated at discrete points. This modeling choice modifies certain aspects of the previously introduced framework and is strongly inspired by the works of Camilli et al.\ \cite{Camili2017} and Evers et al.\ \cite{Evers2015}.
Before proceeding, we provide the definition of a Borel measure.
	
	\begin{definition}[Borel measure]
		Let \((X, \mathcal{T})\) be a topological space and \(\mathcal{B}(X)\) the Borel \(\sigma\)-algebra generated by the topology \(\mathcal{T}\).  
		A \emph{Borel measure} on \(X\) is a measure
		\[
		\nu : \mathcal{B}(X) \to [0, \infty]
		\]
		in the sense of measure theory, i.e., it satisfies:
		\begin{enumerate}
			\item \(\nu(\emptyset) = 0\),
			\item For every sequence of pairwise disjoint sets \((A_i)_{i \in \mathbb{N}} \subset \mathcal{B}(X)\), we have
			\[
			\nu\left( \bigcup_{i=1}^\infty A_i \right) = \sum_{i=1}^\infty \nu(A_i).
			\]
		\end{enumerate}
		The measure \(\nu\) is called \emph{finite} if
		\[
		\nu(X) < \infty.
		\]
		The set of all finite Borel measures on \(X\) is denoted by \(\mathcal{M}(X)\).
		We denote by $\mathcal{M}^+(\mathcal{T})$ the convex cone of
		the positive measures in $\mathcal{M} (\mathcal{T} )$.
	\end{definition}

	The Dirac measure \( \varepsilon_x \), concentrated at a point \( x \in X \), as well as any finite weighted sum of Dirac measures, belongs to the class of finite Borel measures. 
	In the modeling of tram networks, such measures play a crucial role, as we represent a tram as a moving atomic mass on the network. 
	Throughout the following, we refer to such measures as \emph{atomic measures}.
	For $\mu \in \mathcal{M}(\mathcal{T})$ and a bounded measurable function $\varphi: \mathcal{T} \rightarrow \mathbb{R}$ we write
	\begin{align*}
		\langle \mu, \varphi \rangle = \int_\mathcal{T}\varphi d\mu.
	\end{align*}
	Given a Borel measurable vector field $\Phi : \mathcal{T} \rightarrow \mathcal{T}$, the push-forward of the	measure $\mu$ under the action of $\Phi$ is an operation on $\mu$ which produces the new measure $(\Phi \# \mu)$ defined by
	\begin{align*}
		(\Phi \# \mu)(E) = \mu(\Phi^{-1}(E)), \qquad \forall E \in \mathcal{B}(\mathcal{T}).
	\end{align*}
	The Banach space of the bounded and Lipschitz continuous functions $\varphi: \mathcal{T} \rightarrow \mathbb{R}$ is equipped with the norm 
	\begin{align*}
		\| \varphi \|_{BL} = \|\varphi\|_\infty + |\varphi|_L,
	\end{align*}
	where $|\cdot|_L$ denotes the Lipschitz constant. A norm in $\mathcal{M}(\mathcal{T})$ is given by the dual norm of $\| \cdot \|_{BL}$ defined by
	\begin{align*}
		\|\mu\|^*_{BL} = \underset{\varphi \in BL(\mathcal{T}), \|\varphi\|_{BL} \leq 1}{\sup} <\mu,\varphi>.
	\end{align*}
	Note that $\mathcal{M}(\mathcal{T})$ is not complete. However, the space of positive measures $\mathcal{M}^+(\mathcal{T})$ is complete, but not a vector space (multiplication with negative scalars is not contained in $\mathcal{M}^+(\mathcal{T})$), hence not a Banach space. We put our focus on the latter, as we exclusively work with positive passenger masses. Furthermore the topology under consideration in the single edge case is $\mathcal{T} = [0,l_e] \times [0,\infty)$. 
	
	We reconsider the transport problem \eqref{eq: PDEedgeBalanceParticular}--\eqref{eq:InflowPDEedge} on a single edge given by \( [0,l_e] \) and interpret its solution in terms of finite Borel measures \( \nu^e \) on the Cartesian product of the spatial and temporal domains \( [0,l_e] \times [0,\infty) \):
	\begin{align}
	%	\begin{split}
		%	\label{eq:measureProblem}
			\nu^e_t + (w^e\nu^e)_x 
			&= b^e(t,\nu^e(0,t))\,\varepsilon_0(x) 
			- a^e(t,\nu^e(l_e,t))\,\varepsilon_{l_e}(x),
			\quad x \in [0,l_e],~ t \in (0,\infty),~ e \in \mathcal{E},\nonumber \\
			\nu^e_{t=0} 
			&= \nu^e_0 \in \mathcal{M}([0,l_e] \times \{0\}), 
			\qquad x \in [0,l_e],~ e \in \mathcal{E},\label{eq:measureProblem} \\
			\tilde{\nu}^e_{x=0} 
			&= \tilde{\nu}^e_0 \in \mathcal{M}(\{0\} \times [0,\infty)),
			\qquad t \in (0,\infty),~ e \in \mathcal{E}.\nonumber
	%	\end{split}
	\end{align}
	To describe the initial and boundary information in the sense of measures, we project the measure \( \nu^e \) onto the sets \( t = 0 \) and \( x = 0 \). 
	We denote by \( \nu^e_t \) the conditional measure with respect to time, defined by
	\begin{align*}
		\nu^e_t(dx) \otimes dt = \nu^e(dx\,dt),
	\end{align*}
	where \( dt \) denotes the Lebesgue measure in time, and \( \nu^e_t \in \mathcal{M}([0,l_e] \times \{t\}) \) is an atomic measure consisting of a finite sum of Dirac measures in space.
	
	Instead of prescribing the boundary condition at \( x = 0 \) via an inflow condition, we impose it analogously by introducing \( \tilde{\nu}^e_x \), the conditional measure with respect to space, defined by
	\begin{align}
		\label{eq:inflowMeasure}
		\frac{\tilde{\nu}^e_x(dt)}{w^e} \otimes dx = \nu^e(dx\,dt),
	\end{align}
	where \( dx \) is the Lebesgue measure in space, and \( \tilde{\nu}^e_x \in \mathcal{M}(\{x\} \times [0,\infty)) \) is again an atomic measure in time. 
	The denominator \( w^e \) accounts for the compression or stretching of the mass distribution along the temporal axis induced by the velocity field \( w^e \), ensuring that the conditional measure \( \tilde{\nu}^e_x \) accurately represents the total mass located on \( \{x\} \times [0,\infty) \). 
	In the case of atomic boundary masses, this factor has no substantial effect on the solution. 
	Specific forms of \( \tilde{\nu}^{e}_{x=0} \) will be discussed when extending the formulation to the full network in Section~\ref{sec: solNetwork}. 	At $x=l_e$ we have exclusively out-flowing information as we assume $w^e>0$. Therefore $\tilde{\nu}_{l_e}^e(t)$ is implicitly given by the solution of the problem and does not require external information except for the alighting passengers.
	
	For a more detailed discussion of measure spaces and their associated metrics, we refer to Section~3 of the work by Camilli et.al. in~\cite{Camili2017}.

	\subsection{Modeling of passengers at the vertex}
	\label{sec: boardingAlighting}
	For the passenger behavior, we make the following assumptions. 
	At each tram stop located at a vertex, passengers may board or alight from the tram governed by the right-hand side of the PDE. 
	Since trams operate at fixed service frequencies, passengers may have to wait for the next arriving tram. 
	We introduce a time-dependent queue \( q^{e}: [0,\infty) \rightarrow [0,\infty) \) for every outgoing edge \( e \in v^+ \) of a given vertex \( v \in \mathcal{V} \), representing the number of waiting passengers at the origin of edge \( e \). 
	Instead of associating the queue directly with the vertex, we attach it to the outgoing edges. 
	This allows for a clearer distinction at junctions with multiple outgoing connections, since passengers are typically waiting for trams heading in specific directions. 
	%Passengers are assumed to arrive randomly at the tram stop where the inter-arrival times of two passengers are exponentially distributed with parameter $\lambda^e(t)$.
	
	The queue dynamics are driven by two components: the passenger arrival process and the scheduled boarding events. 
	Passenger arrivals for trams heading along edge \( e \in v^+ \) are modeled by a time-inhomogeneous Poisson process \( X^e(t) \) with rate function \( \lambda^e(t) \), representing the expected number of arriving passengers per unit time. 
	At the scheduled stop times, passengers are allowed to board the tram, thereby reducing the queue. 
	The number of passengers boarding at time \( t \) is denoted by \( b^e(t) \). 
	The resulting queue dynamics can thus be expressed as
	\begin{align}
		\label{eq:queue}
		q^e(t) = X^e(t) - \int_0^t b^e(s, \nu^e(0,t))\, ds, 
		\qquad q^e(0) = q_0 \ge 0,
	\end{align}
	where \( X^e(t) \sim \text{Poisson}\!\left(\int_0^t \lambda^e(s)\, ds\right) \) denotes the cumulative number of arrivals up to time \( t \). 
	%The process \( q^e(t) \) is constrained to remain nonnegative, i.e., \( q^e(t) \ge 0 \) (see Lemma~\ref{lem:queue_pos}). 
	
	We assume that the rate function \( \lambda^e \) is piecewise constant on hourly intervals, as a continuously varying rate is difficult to estimate in practice.

	\begin{figure}[ht!]
		\centering
		\includegraphics[width=0.85\linewidth]{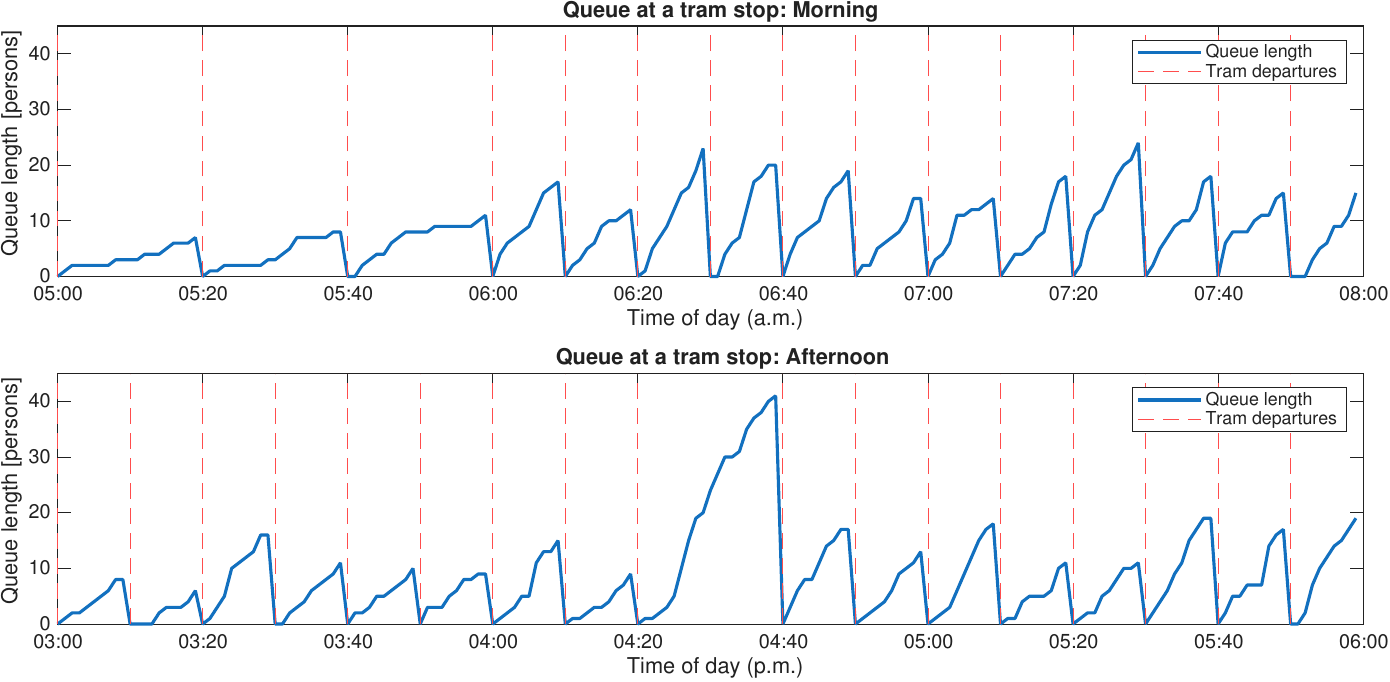}
		\caption{Exemplary evolution of the passenger queue at a tram station during the morning and afternoon periods, together with the corresponding tram departure times.}
		\label{fig:queuePlot}
	\end{figure}
	
	Figure~\ref{fig:queuePlot} illustrates an example of the queue evolution under varying hourly arrival rates and different tram service frequencies. 
	Between 5~and 8~a.m., passenger arrivals increase substantially during the morning peak period. 
	In the simulation, the mean arrival rate rises from approximately 0.1 passengers per minute during early morning hours to about 1.5 passengers per minute between 6~and 9~a.m. 
	To accommodate this increased demand, trams operate at a frequency of 20 minutes from 5~to 6~a.m., which is then reduced to 10 minutes between 6~and 8~a.m. 
	Similarly, during the afternoon peak between 3 and 6~p.m.\, the passenger arrival rate again increases to roughly 1.8 passengers per minute. 
	During this period, a 10-minute service frequency is maintained to ensure a higher service level matching the elevated demand. 
	As an example of possible service disruptions, the cancellation of a tram at 4:30~p.m. leads to a pronounced temporary increase in the queue length.
	
	For simplicity, and when considering the entire network, we assume that the family of arrival processes \( (X^e(t))_{e \in \mathcal{E}} \) is pairwise independent. 
	In real-world applications, this assumption is generally violated, since passenger arrivals may exhibit correlations across stations due to external factors such as weather conditions, special events, or public holidays. 
	A more comprehensive treatment of these correlations would require a deeper stochastic analysis beyond the present scope. Between two tram departures, passenger arrivals are assumed to be uniformly distributed over the entire inter-departure interval. 
	Although, in practice, arrivals are often more concentrated shortly before a tram’s departure, we deliberately abstract from such temporal variations. 
	For the purposes of our analysis, the assumption of uniformly distributed arrivals provides an adequate approximation.

	\paragraph{Boarding and alighting passengers}
	
The number of passengers boarding the tram at time \(t\) at the stop associated with edge \(e\) is given by
\begin{align}
	\label{eq:boarding}
	b^e(t, \nu^e(0,t))
	= \mathbbm{1}_{\{\tau^e(t) > 0\}}
	\min\left\{ q^e(t^-),\, \tau^e(t) - \nu^{e}(0,t) \right\}.
\end{align}
We distinguish two cases.
If, after the alighting process, sufficient capacity remains in the tram, the entire queue \(q^e(t^-)\) of waiting passengers for edge \(e\) can board. 
Here we define
\[
t^- := \lim_{s \nearrow t} s
\]
and use \(q^e(t^-)\) to denote the left limit of the queue length prior to boarding.
Otherwise, boarding continues only until the tram reaches its capacity limit.

Note that if \(\tau^e(t) = 0\), that is, if according to the admissible schedule no tram is present at the stop at time \(t\), it follows immediately that \(b^e(t) = 0\).  
Moreover, the number of boarding passengers is independent of the number of passengers already inside the tram, represented by \(\nu^e(0,t)\), as long as the overall capacity is not exceeded.

The number of passengers alighting from the tram at a given stop is modeled as a fraction of the passengers currently onboard. 
Instead of associating the alighting process with the outgoing edges of a vertex, we attach it to the incoming edge, which yields
\begin{align}
	\label{eq:alightingDiscrete}
	a^e(t, \nu^e(l_e,t))
	= \mathbbm{1}_{\left\{
		\tau^e\!\left(t - \frac{l_e}{w^e}\right) > 0
		\right\}}
	r_a^e(t)\, \nu^{e}(l_e,t),
\end{align}
where \( r_a^e(t) \in [0,1] \) denotes the deterministic alighting fraction. 
This fraction may depend on time and on the specific edge \(e\).% and is determined from available data.
The indicator in \eqref{eq:alightingDiscrete} ensures that no passengers alight at times when no tram is present at the stop.
At terminal stops, i.e., for edges \(e\) satisfying \( \delta_t^v(e) = \emptyset \), we set \( r_a^e(t) = 1 \), corresponding to all passengers leaving the tram.

Modeling alighting in this way and not using absolute values offers several advantages.
First, the number of alighting passengers is directly linked to the number of passengers currently onboard, thereby reflecting fluctuations in boarding activity.
Second, the formulation naturally captures gradual disembarkation as the tram approaches its terminal station.
This ensures internal consistency between the boarding and alighting mechanisms within the overall tram-network model that will be introduced later.
	
\begin{lemma}
	\label{lem:queue_pos}
	Let the arrival rates \(\lambda^e(t)\) of the Poisson process be uniformly bounded, and assume that the capacities prescribed by the admissible tram schedule \(\tau^e(t)\) are uniformly bounded as well. Then the queue \(q^e(t)\) defined in Equation~\eqref{eq:queue} is well-defined and nonnegative for all \(t \ge 0\).
\end{lemma}
	\begin{proof}
		Let \(T>0\) be arbitrary. By assumption the arrival rates \(\lambda^e(t)\) are uniformly bounded, hence the inhomogeneous Poisson process \(X^e(t)\) is a.s.\ finite on \([0,T]\).  Since \(b^e(t,\nu^e(0,t))\ge 0\) for all \(t\), we have for every \(t\in[0,T]\)
		\[
		q^e(t)=X^e(t)-\int_0^t b^e(s,\nu^e(0,s))\,\mathrm ds \le X^e(t) \le X^e(T),
		\]
		so \(\sup_{0\le t\le T} q^e(t)\le X^e(T)<\infty\) a.s.
		Furthermore we a.s. get
		\begin{align*}
			\int_0^t b^e(s, \nu^e(0,s)) ds &= \int_0^t \mathbbm{1}_{\lbrace\tau^e(s)>0\rbrace} \min\lbrace  q^e(s) , \tau^e(s)-\nu^{e}(0,s)   \rbrace ds\\
			&= \sum_{\lbrace s \geq0~ :~ \tau^e(s)>0, s\leq t\rbrace } \min\lbrace  q^e(s) , \tau^e(s)-\nu^{e}(0,s) \rbrace\\
			&\leq \sum_{\lbrace s \geq0~ :~ \tau^e(s)>0, s\leq t\rbrace } q^e(s)<\infty,
		\end{align*}
		as the sum consists of only finitely many points, due to the finitely many times at which a tram stops.
		
		It remains to show the positivity of the queue. Denote the elements of 
		\(\{ s \ge 0 : \tau^e(s) > 0, s \le t \}\) by \(s_1 < s_2 < \dots < s_k\). 
		For \(t < s_1\), we have \(\int_0^t b^e(s,\nu^e)\, ds = 0\), and thus 
		\(q^e(t) = X^e(t) - 0 \ge 0\).
		At \(t = s_1\) we have
		\[
		\int_0^{s_1} b^e(s,\nu^e(0,s))\, ds = b^e(s_1, \nu^e(0,s_1))
		= \min \{ q^e(s_1^-), \tau^e(s_1) - \nu^e(0,s_1) \} \le q^e(s_1^-),
		\]
		so \(q^e(s_1) \ge 0\).	Now assume that \(q^e(t) \ge 0\) for all \(t \in [0, s_i]\) almost surely. For \(t \in (s_i, s_{i+1})\), we have
		\[
		\int_0^t b^e(s, \nu^e(0,s))\, ds = \sum_{k=1}^i b^e(s_k, \nu^e(0,s_k)) \le q^e(s_i)
		\]
		by the induction hypothesis. Therefore, we get a.s.
		\[
		q^e(t) = X^e(t) - \int_0^t b^e(s,\nu^e(0,s))\, ds \ge X^e(t) - q^e(s_i) \ge 0.
		\]
		At \(t = s_{i+1}\), we have
		\begin{align*}
			&\int_{(s_i, s_{i+1}]} b^e(s, \nu^e(0,s))\, ds 
			= b^e(s_{i+1}, \nu^e(0,s_{i+1})) \\
			&= \min \{ q^e(s_{i+1}^-), \tau^e(s_{i+1}) - \nu^e(0,s_{i+1}) \} \le q^e(s_{i+1}^-),
		\end{align*}
		so \(q^e(s_{i+1}) \ge 0\) almost surely.
	
		Repeating this argument inductively over all boarding times completes the proof.
		
	\end{proof}

By the construction of the boarding and alighting passengers, the function 
\( g(x,t,\nu^e) \) in Equation~\eqref{eq: PDEedgeBalanceParticular} is given as the sum of two families of weighted Dirac measures, 
whose masses are located at the initial and terminal points of the edge, 
and at those times \( t \) for which \( \tau^e(t) > 0 \). Consequently, we have
\[
b^e \in \mathcal{M}(\{0\} \times [0,\infty)) 
\quad \text{and} \quad 
a^e \in \mathcal{M}(\{l_e\} \times [0,\infty)).
\]
We denote by \( \mathcal{T}(e) \) the finite set of time points at which \( \tau^e \) is nonzero, 
and by \( |\mathcal{T}(e)| \) its cardinality. 
The latter corresponds to the number of trams stopping at the vertex \( v(e) \) 
and departing along edge \( e \) within the considered time horizon \( [0,T] \).

The measure of boarding passengers along edge \( e \) can therefore be represented as a finite sum of weighted Dirac measures,
\begin{align*}
	b^e = \sum_{i=1}^{|\mathcal{T}(e)|} \beta^e_i\, \varepsilon_{t_i},
\end{align*}
where \( t_i \in \mathcal{T}(e) \) denote the boarding times and 
\( \beta^e_i \in \mathbb{R}_{\ge 0} \) the corresponding boarding masses, i.e., 
the number of passengers boarding at time \( t_i \). 
The values \( \beta^e_i \) are determined by the passenger arrival Poisson process introduced in Equation~\eqref{eq:queue} and the number of boarding passengers according to Equation \eqref{eq:boarding}.

Similarly, the alighting process can be expressed as
\begin{align*}
	a^e = \sum_{i=1}^{|\mathcal{T}(e)|} 
	\alpha^e_i\, 
	\varepsilon_{t_i + \frac{l_e}{w^e}},
\end{align*}
where \( \alpha^e_i \in \mathbb{R}_{\ge 0} \) denote the corresponding alighting masses, 
representing the number of passengers disembarking at time 
\( t_i + \frac{l_e}{w^e} \), 
that is, when the tram reaches the end of edge \( e \). 

This atomic representation of boarding and alighting processes provides a natural connection 
between the continuous network dynamics and the discrete tram schedule, 
while maintaining a consistent measure-theoretic framework that is further elaborated in the next section.

	\subsection{Weak measure-valued solution on one edge}
	%{\color{green} sowas schreiben, wie die Lösung setzt sich auf den Diracs aus den Anfangsdaten und den Randdaten zusammen und diese werden dann transportiert. Anfangdatum über gammas einbauen?}

	We set up the definition of a measure-valued solution. For brevity, we omit the explicit dependence on $\nu^e(0,t)$ and $\nu^e(l_e,t)$ in $b^e$ and $a^e$ respectively.
	Starting from the problem
	\begin{align*}
		\nu^e_t(x,t) + (w^e \nu^e(x,t))_x = b^e(t)\varepsilon_0(x) - a^e(t)\varepsilon_{l_e}(x),
	\end{align*}
	we multiply with a test function $\varphi^e \in C_0^1([0,l_e] \times [0,\infty))$ which in particular belongs to the class of bounded Lipschitz functions, and integrate over the space-time domain.
	Integration by parts, together with the properties of the test functions, then leads to the following definition of a weak measure-valued solution on a single edge.

	\begin{definition}
		\label{def: weakMeasureValuedSolutionOneEdge}
		A weak measure-valued solution to Equation \eqref{eq:measureProblem} is given by a Borel measure  $\nu^e \in \mathcal{M}([0,l_e] \times [0,\infty))$ satisfying
		\begin{align*}
			&\int_0^{l_e}\int_0^\infty \varphi^e_t(x,t) + w^e\varphi^e_x(x,t)d\nu^e(x,t) \\
			&= \int_0^{l_e} \varphi^e(x,0)d\nu^e_{t=0}(x) - \int_0^\infty \varphi^e(l_e,t)d\tilde{\nu}^e_{x=l_e}(t) -\varphi^e(0,t)d\tilde{\nu}^e_{x=0}(t)\\
			& ~~+  \int_0^\infty \varphi^e(0,t) db^e(t) - \int_0^\infty  \varphi^e(l_e,t)da^e(t)
		\end{align*}
		for all test functions $\varphi^e \in C_0^1([0,l_e] \times [0,\infty))$. 
	\end{definition}
The outflow measure $\tilde{\nu}^e_{x=l_e}$ is induced by the system through the strictly positive velocity $w^e$. We slightly modify the construction of the push-forward and let it act exclusively on the spatial variable, while treating time as a parameter. The push-forward solution (or trajectory) for our problem is denoted by
	\begin{align}
		\label{eq:PushForward}
		\Phi_t^e(x_0) = x_0 + t w^e,
	\end{align}
	where we start in position $x_0$. The push-forward of the
	measure $\nu^e(A,s)$ under the action of $\Phi^e_t$ is an operation on $\nu^e$ which produces the new measure $\Phi^e_t \# \nu^e \in \mathcal{M}([0,l_e])$ and is defined by
	\begin{align*}
		(\Phi^e_t \#\nu^e)(A,s) = \nu^e((\Phi^e_t)^{-1}(A),s),
	\end{align*}
	for any Borel set $A \in \mathcal{B}([0,l_e])$. Due to the easy linear structure for a one-point mass $\lbrace x \rbrace \in [0,l_e]$ we can write
	\begin{align}
		\label{eq:PushforwardLinear}
		\nu^e(x,s+t) = (\Phi^e_t \#\nu^e)(\lbrace x \rbrace,s) = \nu^e((\Phi^e_t)^{-1}(\lbrace x \rbrace),s) = \nu^e(x-tw^e,s).
	\end{align}
Similar problems to the one in our investigation have been considered in literature. In the work by Camilli, et.al.\ \cite{Camili2017} in Theorem 4.2 and Theorem 4.3 they construct the unique measure-valued solution to the problem \eqref{eq:measureProblem} for a zero right-hand side. 

For solutions to problems of our type \eqref{eq:measureProblem} with non-zero right-hand side we consider the work of Evers et.al. \cite{Evers2015}. %Instead of working with the definition of measure-valued solutions using the integration over test functions, they prefer to work in a measure-theoretic framework using push-forward measures. This is an equivalent way to understand measure-valued solutions to linear transport equations, but will not further be discussed here (see \cite{Camili2017}).
In particular, they consider right-hand sides that are nonzero only at the boundary. 
Proposition 3.3 in \cite{Evers2015} shows that if the right-hand side $g(x,t,\nu^e)$ can be decomposed into $g(x,t,\nu^e(x,t)) = f(x)\cdot \nu^e(x,t)$ with a piecewise bounded Lipschitz function and $w^e(x)>0$, then there exists a solution to \eqref{eq: PDEedgeBalanceParticular} in the sense of a push-forward measure. In our case we can find a decomposition, but $f$ is clearly not Lipschitz as it non-zero in two individual points. But following Lemma 4.1 in \cite{Evers2015}, if we can find a sequence $(f_n)_{n\in \mathbb{N}}$ of piecewise bounded Lipschitz functions that converges to $f$ pointwise in a suitable norm ($BL^*$-norm).
	
Inspired and following those two ideas we provide the unique solution measure and prove that it satisfies Definition \ref{def: weakMeasureValuedSolutionOneEdge}.
The solution to our problem is given by four components, the initial information $\nu^{e,1}$, the boundary information from the left $\nu^{e,2}$ (due to positive speed) and the boarding $\nu^{e,3}$ and alighting passengers $\nu^{e,4}$. 
For all components, we can construct solutions separately and then add them, 
since the dynamics along each edge are governed by a linear partial differential equation. For $A \in \mathcal{B}([0,l_e])$ we have
\begin{align*}
	\nu^e(A,t) &= \nu^{e,1}(A,t) + \nu^{e,2}(A,t) + \nu^{e,3}(A,t) - \nu^{e,4}(A,t)\\
	&= \int_0^{l_e} \varepsilon_{\Phi_t^e(y,0)}(A) d\nu^e_{t=0}(y) + \int_0^t  \varepsilon_{\Phi_t^e(0,s)}(A)d\tilde{\nu}^e_{x=0}(s) \\
	&~~~+ \int_0^t  \varepsilon_{\Phi_t^e(0,s)}(A)db^e(s) - \int_0^t  \varepsilon_{\Phi_t^e(l_e,s)}(A)da^e(s).
\end{align*}
As we work on a bounded spatial domain, the integrals might not be well-suited. We introduce the quantities 
\begin{align*}
	\theta^e(x) &= \inf \lbrace t\geq 0~:~ \Phi_t^e(x,0) = 1 \rbrace,\\
	\sigma^e(s) &= \inf \lbrace t\geq s~:~ \Phi_t^e(0,s) = 1 \rbrace.
\end{align*}
to restrict the integration to the area of interest. As the velocity $w^e$ does not depend on the measure itself, the functions $\theta^e, \sigma^e$ are one-to-one and have well-defined inverses.
The solution measure is described by
\begin{align}
	\label{eq:SolutionMeasure_mit4}
	\begin{split}
		\nu^e(A,t) &= \nu^{e,1}(A,t) + \nu^{e,2}(A,t) + \nu^{e,3}(A,t) - \nu^{e,4}(A,t)\\
		&= \int_0^{\max \lbrace 0, (\theta^e)^{-1}(t)   \rbrace} \varepsilon_{\Phi_t^e(y,0)}(A) d\nu^e_{t=0}(y) + \int_{\max \lbrace 0, (\sigma^e)^{-1}(t) \rbrace}^t  \varepsilon_{\Phi_t^e(0,s)}(A)d\tilde{\nu}^e_{x=0}(s) \\
		&~~~+ \int_{\max \lbrace 0, (\sigma^e)^{-1}(t) \rbrace}^t  \varepsilon_{\Phi_t^e(0,s)}(A)db^e(s) - \int_{t}^t  \varepsilon_{\Phi_t^e(l_e,s)}(A)da^e(s). %\varepsilon_{l_e}(A)a^e(t).
	\end{split}
\end{align}
Note that the last term stemming from $\nu^{e,4}$ vanishes as it exclusively affects the outflow and not the solution on the edge. Thus it remains
\begin{align}
	\begin{split}
		\label{eq:SolutionMeasure}
		\nu^e(A,t) &= \int_0^{\max \lbrace 0, (\theta^e)^{-1}(t)   \rbrace} \varepsilon_{\Phi_t^e(y,0)}(A) d\nu^e_{t=0}(y) + \int_{\max \lbrace 0, (\sigma^e)^{-1}(t) \rbrace}^t  \varepsilon_{\Phi_t^e(0,s)}(A)d\tilde{\nu}^e_{x=0}(s) \\
		&~~~+ \int_{\max \lbrace 0, (\sigma^e)^{-1}(t) \rbrace}^t  \varepsilon_{\Phi_t^e(0,s)}(A)db^e(s).
	\end{split}
\end{align}
%	\textcolor{red}{Hier muss ich jetzt schauen, ob die Integralgrenzen passen! Nacharbeit erforderlich}
We can show that this measure stays positive and bounded.

	\begin{lemma}
	Without loss of generality, assume that we start in a system without traveling passengers, i.e.\ $\nu^e_{t=0}=0$. For any $x \in [0,l_e], t\geq 0, e \in \mathcal{E}$ it holds $0\leq \nu^e(x,t) \leq \tau^{e}\left(t-\frac{x}{w^e}\right)$. 
\end{lemma}
\begin{proof}
	In the representation of Equation \eqref{eq:SolutionMeasure} the first integral vanishes by the assumption of an empty system at $t=0$. For the boarding passengers we have according to Equation \eqref{eq:boarding}
	\begin{align*}
		b^e(t,\tilde{\nu}_{x=0}^e(t)) &= \mathbbm{1}_{\{\tau^e(t) > 0\}} 
		\min\left\{ q^e(t^-),\, \tau^e(t) - \tilde{\nu}_{x=0}^e(t) \right\}\\
		&\leq \tau^e(t) - \tilde{\nu}_{x=0}^e(t).
	\end{align*}
		Then we get
	\begin{align*}
		\nu^e(0,t) = \tilde{\nu}_{x=0}^e(t) + b^e(t) \leq \tilde{\nu}_{x=0}^e(t) + \tau^e(t) - \tilde{\nu}_{x=0}^e(t) = \tau^e(t).
	\end{align*}
	By the push-forward description in Equation \eqref{eq:PushforwardLinear} we have
	\begin{align*}
		\nu^e(x,t) = \begin{cases}
			\nu^e\left(0,t-\frac{x}{w^e}\right), & t-\frac{x}{w^e}\geq 0\\ 0, & t-\frac{x}{w^e}<0
		\end{cases}~~ \leq \tau^e\left(t-\frac{x}{w^e}\right).
	\end{align*}
	with the convention that $\tau^e(t)$ is set to 0 for $t<0$. 
	
	As the alighting passengers reduce the traveling passengers according to Equation \eqref{eq:alightingDiscrete} by a share $0\leq r_a^e(t)\leq1$, we obtain that
	\begin{align*}
		\tilde{\nu}^e_{x=l_e}(t) &= \left(1-\mathbbm{1}_{\tau^e\left(t-\frac{l_e}{w^e}\right)}r_a^e(t)\right)\nu^e(l_e,t)\\
		&\leq \left(1-\mathbbm{1}_{\tau^e\left(t-\frac{l_e}{w^e}\right)}r_a^e(t)\right) \tau^e\left(t-\frac{l_e}{w^e}\right) \leq \tau^e\left(t-\frac{l_e}{w^e}\right).
	\end{align*}
	By conservation of mass and the coupling conditions at the junction we get that
	\begin{align*}
		\tau^e(t) - \tilde{\nu}_{x=0}^e(t)\geq 0.
	\end{align*}
	By Lemma \ref{lem:queue_pos} we know that the queue $q^e$ remains nonnegative, therefore $b^e(t)\geq 0$. Thus within the edge and ignoring potential inflow we directly also have $\nu^e(x,t)\geq 0$. Building on this assumption for the outflow we get
	\begin{align*}
		\tilde{\nu}^e_{x=l_e}(t) = \underbrace{\left(1-\mathbbm{1}_{\tau^e\left(t-\frac{l_e}{w^e}\right)}r_a^e(t)\right)}_{\geq 0}\underbrace{\nu^e(l_e,t)}_{\geq 0} \geq 0.
	\end{align*}
	Again by mass conservation also for all inflows we get $\tilde{\nu}^e_{x=0}(t)\geq 0$. 
	We get 
	\begin{align*}
		\nu^e(x,t) = \begin{cases}
			\nu^e\left(0,t-\frac{x}{w^e}\right) \geq  \tilde{\nu}_{x=0}^e\left(t-\frac{x}{w^e}\right) + b^e\left(t-\frac{x}{w^e}\right), & t-\frac{x}{w^e}\geq 0\\ 0, & t-\frac{x}{w^e}<0
		\end{cases} ~~ \geq 0.
	\end{align*}
\end{proof}
We can show that this is exactly the unique measure-valued solution to the tram problem on one edge.
	\begin{theorem}
		\label{thm:UniqueSolutionOneEdge}
		The positive Borel measure described by Equation \eqref{eq:SolutionMeasure} is the unique measure-valued solution to the tram network problem \eqref{eq:measureProblem} on one edge according to Definition \ref{def: weakMeasureValuedSolutionOneEdge}.
	\end{theorem}
We provide the proof in the Appendix \ref{app: appendix}.

	\subsection{Construction of solutions on the network}
	\label{sec: solNetwork}
	Based on the solutions found on one network edge, we extend the solution to the entire network. The network problem is given by
	\begin{align}
		\begin{split}
			\label{eq:NetworkProblem}
			\nu_t^e + (w^e\nu^e)_x &= b^e(t)\varepsilon_0(x) - a^e(t)\varepsilon_{l_e}(x), \qquad x \in [0,l_e],~ t \in (0,\infty), ~e\in\mathcal{E},\\
			\nu_{t=0}^e &= \nu_0^e, \qquad x \in [0,l_e],~ e\in\mathcal{E},\\
			\tilde{\nu}_{x=0}^e &= \tilde{\nu}^{(\delta^{v(e)}_t)^{-1}(e)}_{x=l_{(\delta^{v(e)}_t)^{-1}(e)}}, \qquad t \in (0,\infty), ~ e\in \mathcal{E}.%\\
			%  \tilde{\nu}_{x=l_e}^e &= w(l_e)\nu^e(l_e,\cdot)-a^e, \qquad %t\in(0,\infty), ~e\in \mathcal{E}
		\end{split}
	\end{align}
	In case $(\delta^{v(e)}_t)^{-1}(e)$ is the empty set, then the contribution containing the $\delta_t^v$-function is considered to be zero.
	For the condition at the vertices we make use of the particular tram schedule definition and the functions $\delta_t^v$. In Section~\ref{sec: admissibleTramSchedule}, we have already discussed the structure of \(\delta_t^v\), which governs the flow of tram passengers across a vertex. Based on this we provide a detailed description of the vertex conditions for the in- and outflowing measures in the exemplary cases of 1--1, 1--2, 2--1, and, more generally, \( m\text{--}n \) junctions.

\paragraph{1--1 junction.}
Consider a vertex \( v \in \mathcal{V} \) with exactly one incoming edge \( e^- \in v^- \) and one outgoing edge \( e^+ \in v^+ \). 
The boundary condition reads
\begin{align}
	\label{eq:coupling1-1}
	\tilde{\nu}^{e^+}_{x=0}(t) =
	\begin{cases}
		\tilde{\nu}^{e^-}_{x=l_{e^-}}(t) , & \text{if } \delta_t^v(e^-) = e^+, \\[0.3em]
		0, & \text{if } \delta_t^v(e^-) = \emptyset.
	\end{cases}
\end{align}
%\begin{align}
%	\label{eq:coupling1-1}
%	\tilde{\nu}^{e^+}_{x=0^+}(t) =
%	\begin{cases}
%		\tilde{\nu}^{e^-}_{x=l_{e^-}}(t) + b^{e^+}(t) , & \text{if } \delta_t^v(e^-) = e^+, \\[0.3em]
%		b^{e^+}(t), & \text{if } \delta_t^v(e^-) = \emptyset.
%	\end{cases}
%\end{align}
%Here, \(\tilde{\nu}^{e^+}_{x=0^+}(t)\) denotes the right-hand limit towards \(x=0\), i.e., the measure after passengers have boarded, while \(\tilde{\nu}^{e^-}_{x=l_{e^-}}(t)\) denotes the outflowing passengers from $e^-$ after the alighting.
If no tram arrives at time \(t\) that continues to \(e^+\), there is no information transported via the vertex. Nevertheless a new trip may originate in the vertex, which is then reflected by boarding passengers and the measure $b^{e^+}(t)$.

\paragraph{1--2 junction.}
In this case, the vertex \(v\) has one incoming edge \( e^- \in v^- \) and two outgoing edges \( e_1^+, e_2^+ \in v^+ \). 
Using the definition of \(\delta_t^v\), we obtain
\begin{align}
	\label{eq:coupling1-2}
	\tilde{\nu}^{e^+_i}_{x=0}(t) =
	\begin{cases}
		\tilde{\nu}^{e^-}_{x=l_{e^-}}(t)  , & \text{if } (\delta_t^v)^{-1}(e_i^+) = e^-, \\[0.3em]
		0, & \text{if } (\delta_t^v)^{-1}(e_i^+) = \emptyset,
	\end{cases}
	\qquad i = 1,2.
\end{align}
%\begin{align}
%	\label{eq:coupling1-2}
%	\tilde{\nu}^{e^+_i}_{x=0}(t) =
%	\begin{cases}
%		\tilde{\nu}^{e^-}_{x=l_{e}}(t) + b^{e^+_i}(t) , & \text{if } (\delta_t^v)^{-1}(e_i^+) = e^-, \\[0.3em]
%		b^{e^+_i}(t), & \text{if } (\delta_t^v)^{-1}(e_i^+) = \emptyset,
%	\end{cases}
%	\qquad i = 1,2.
%\end{align}
As mentioned previously, the inflow of passengers entering the vertex cannot split into multiple outgoing fluxes. 
Hence, for any fixed \(t\), the first case cannot occur simultaneously for both outgoing edges.

In Figure \ref{fig:toyNetwork} we show an exemplary 1-1-1-1-1 network with 5 tram stops located at $x=0,x=3,x=5,x=9,x=10$. The trams starts at $x=0$ at the times $2,12,22,32,42,52$ (every ten minutes). One can see how the masses are transported in space and time. The passenger loads are described by the color of the lines. Between the tram stops the number of passengers does not change, since boarding and alighting is only possible at the tram stops. Until the tram stop 3 at $x=5$ the trams are rather empty. On the section between tram stop 3 and 4 many passengers travel in the tram, which can be observed by the more yellow or even red colors. Particularly the tram starting at $t=42$ is highly crowded. %The spatial component of the intervals was chosen continuously and not reset to zero at the beginning of each edge. 

\begin{figure}[ht!]
\centering
\includegraphics[width=0.8\linewidth]{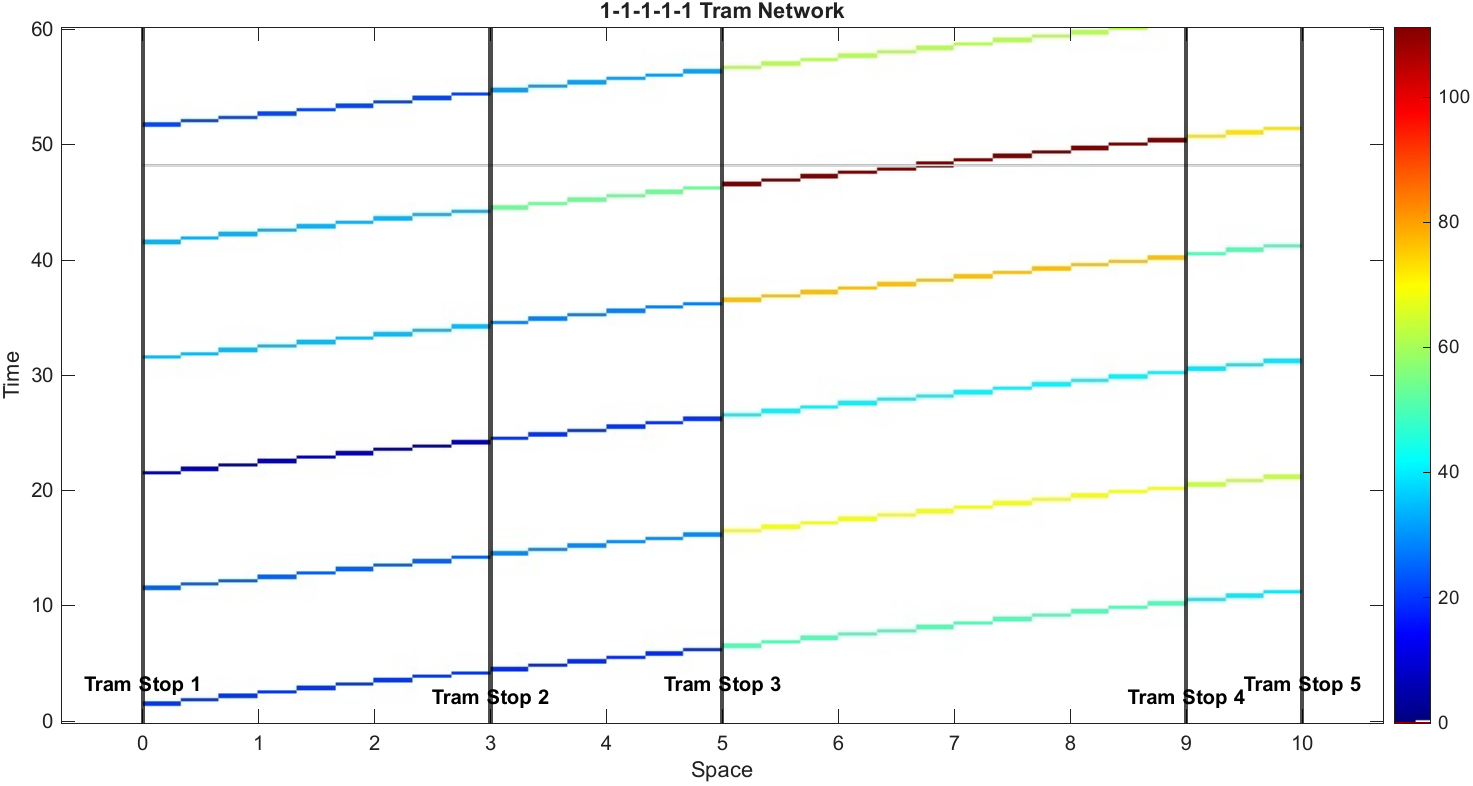}
\caption{Evolution of trams in a small 1-1-1-1-1 tram network. The lines show the movement of the tram in space and time and the color represents the number of traveling passengers. Tram stops are marked with black vertical lines.}
\label{fig:toyNetwork}
\end{figure}

\paragraph{2--1 junction.}
Here, the vertex \(v\) has two incoming edges \( e_1^-, e_2^- \in v^- \) and one outgoing edge \( e^+ \in v^+ \). 
The boundary condition for the outgoing edge is defined as
\begin{align}
	\label{eq:coupling2-1}
	\tilde{\nu}^{e^+}_{x=0}(t) =
	\begin{cases}
		\tilde{\nu}^{e^-_i}_{x=l_{e^-_i}}(t), & \text{if } \delta_t^v(e^-_i) = e^+, \\[0.3em]
		0, & \text{if } \delta_t^v(e^-_1) = \delta_t^v(e^-_2) = \emptyset.
	\end{cases}
\end{align}
%\begin{align}
%	\label{eq:coupling2-1}
%	\tilde{\nu}^{e^+}_{x=0}(t) =
%	\begin{cases}
%		\tilde{\nu}^{e^-_i}_{x=l_{e^-_i}}(t) + b^{e^+}(t) , & \text{if } \delta_t^v(e^-_i) = e^+, \\[0.3em]
%		b^{e^+}(t), & \text{if } \delta_t^v(e^-_1) = \delta_t^v(e^-_2) = \emptyset.
%	\end{cases}
%\end{align}
By construction of the admissible solution, it is not possible that both \(\delta_t^v(e^-_1)\) and \(\delta_t^v(e^-_2)\) are non-empty simultaneously.

\paragraph{m--n junction.}
\label{eq:m-nconditionmeasure}
Finally, the boundary condition at an inner vertex can be generalized for \(m\in \mathbb{N}\) incoming edges \( e^-_1, \dots, e^-_m \) 
and \(n\in \mathbb{N}\) outgoing edges \( e^+_1, \dots, e^+_n \) as
\begin{align}
	\label{eq:coupling2-2}
	\tilde{\nu}^{e^+}_{x=0}(t) =
	\begin{cases}
		\tilde{\nu}^{e^-_i}_{x=l_{e^-_i}}(t) , & \text{if } \delta_t^v(e^-_i) = e^+,~ i = 1, \dots, m, \\[0.3em]
		0, & \text{if } \not\exists\, i \in \{1, \dots, m\} \text{ such that } \delta_t^v(e^-_i) = e^+.
	\end{cases}
\end{align}
By construction, there can be at most one \(i \in \{1, \dots, m\}\) such that \(\delta_t^v(e^-_i) = e^+\) holds.
	
	This allows us to generalize the formulation of a measure-valued solution to the network. Therefore, we consider test functions $\varphi$ on the entire network such that $\varphi \in C_0^1(\mathcal{E} \times [0,\infty))$ and we denote by $\varphi^e$ the restriction of $\varphi$ on edge $e$. Then the weak formulation by summing over all edges in the network is given by 
	\begin{align*}
		\sum_{e\in\mathcal{E}}&\int_0^\infty\int_0^{l_e} \left( \varphi^e_t(x,t) +  w^e\varphi^e_x(x,t) \right) d\nu^e(x,t) \\
		&~~+ \sum_{e\in\mathcal{E}}\int_0^{l_e} \varphi^e(x,0)d\nu^e_0(x)\\
		&~~ + \sum_{e\in\mathcal{E}} \int_0^\infty \varphi^e(0,t)d\tilde{\nu}_{x=0}^e(t) - \int_0^\infty \varphi^e(l_e,t)d\tilde{\nu}_{x=l_e}^e(t)  \\
	%	& ~~+ \sum_{e\in\mathcal{E}} \int_0^\infty b^e(t,\nu^e(0,t))\phi^e(0,t) \phantom{- a^e(t,\nu^e(l_e,t))\phi^e(l_e,t)}dt= 0.\\
		& ~~+ \sum_{e\in\mathcal{E}} \int_0^\infty \varphi^e(0,t)db^e(t)= 0.
	\end{align*}
	The second last line can be further evaluated. We ensure flux conservation at the junction after alighting and before boarding of passengers. 
	Therefore by using continuity of the test functions at the inner vertices $v \in \mathcal{V}$ with ingoing edges $e^- \in \mathcal{E}$ and outgoing edges $e^+ \in \mathcal{E}$ we get for all $t\geq 0$
	\begin{align*}
		\varphi^{e^-}(l_e,t) = \varphi^{e^+}(0,t).
	\end{align*}
	This implies that the in- and outflowing conditions at the vertices cancel out 
	\begin{align*}
		&\sum_{e\in\mathcal{E}} \int_0^\infty \varphi^e(0,t)d\tilde{\nu}_0^e(t) - \int_0^\infty \varphi^e(l_e,t)d\tilde{\nu}_{l_e}^e(t)   =0.
	\end{align*}
	%as all passengers alight the tram at the terminal stop. {\color{red} Aktuell noch ein Problem mit Bahnen die mittendrin enden!}
	At all inner vertices apart from the boarding and alighting dynamics, we conserve the transported passengers. Thus the only remaining quantities are the boarding passengers (associated with the beginning of every edge) and the alighting passengers (associated with the end of every edge). For the source vertices we assume that there is no external inflow into the tram system, which corresponds to trams which are empty before the first passenger boards. %The switch of alighting passengers to the end of the edges is possible due to the continuity of the test functions on the graph. This way of considering avoids additional term for the alighting of passengers at the terminal stop, if alighting would also be associated to the beginning of an edge and there is no successor edge for a terminating tram.
	
	This leads to the definition of the measure-valued solution for the network problem.
	\begin{definition}
		\label{def: measureSolutionNetwork}
		A family of finite Borel measure $\nu = (\nu^e)_{e \in \mathcal{E}}$ such that any $\nu^e \in \mathcal{M}([0,l_e]\times[0,\infty))$ is called a measure-valued solution of \eqref{eq:NetworkProblem}, if for all test functions $\varphi^e \in C^1(\mathcal{E}\times[0,\infty))$
		\begin{align}
			\begin{split}
				\label{eq:WeakFormulation2}
				&\int_0^\infty\int_\mathcal{E} \left( \varphi_t(x,t) + w^e\varphi_x(x,t) \right) d\nu(x,t) + \int_\mathcal{E} \varphi(x,0)d\nu_0(x)\\
				&~~~ + \sum_{e \in \mathcal{E}} \left(\int_0^\infty \varphi^e(0,t)db^e(t, \nu^e(0,t)) - \int_0^\infty \varphi^e(l_e,t)da^e(t, \nu^e(l_e,t))\right)  = 0.%
			\end{split}
		\end{align}
	\end{definition}
	
	For the construction of the solution on the network we can use the solutions for one edge presented in Equation \eqref{eq:SolutionMeasure}. This is, we assume that for any $e\in \mathcal{E}$ the solution measures $\nu^e$ on the edges fulfill Definition \ref{def: weakMeasureValuedSolutionOneEdge}. At the vertices we ensure the the in- and outflow measures are defined as in Equations \eqref{eq:coupling1-1}--\eqref{eq:coupling2-2}. 
	Following the arguments in the work of Camili et.al.\ in Theorem 5.1 \cite{Camili2017} %the solution on the edges from equation \eqref{eq: SolutionEdge} together with the coupling conditions in equation \eqref{eq:m-nconditionmeasure} is the unique measure-valued solution of \eqref{eq:NetworkProblem}.
	allows us to state the following existence and uniqueness theorem.
		\begin{theorem}
		\label{thm:UniqueSolutionNetwork}
		Given a family of functions modeling the boarding and alighting passengers $(b^e(t))_{e \in \mathcal{E}}, (a^e(t))_{e \in \mathcal{E}} \in \mathcal{M}([0,\infty))$ and an initial state $(\nu^e_{t=0})_{e \in \mathcal{E}} \in \mathcal{M}(\mathcal{E})$. Furthermore we assume that the coupling conditions at the vertices are given by \eqref{eq:coupling1-1}--\eqref{eq:coupling2-2}. 
		The family of positive Borel measure described by Equation \eqref{eq:SolutionMeasure} is the unique measure-valued solution to the tram network problem \eqref{eq:NetworkProblem} according to Definition \ref{def: measureSolutionNetwork}.
	\end{theorem}

	\begin{remark}
		In comparison to the work in \cite{Camili2017} the conditions at the vertex are different to our case. Instead of general distribution parameters, where mass can be split at the vertex, for us it is only possible to move the entire mass to exactly one outgoing edge (taking into account potential boarding and alighting of passengers). The basis for this is given by the functions $\delta_t^v$ in Definition \ref{def: admissibleSchedule}. On the other hand, in our work it is possible that mass appears or disappears by boarding and alighting passengers and we do not have external in- or outflows to the network apart from the passengers boarding or alighting. This passenger dynamics occurs at every vertex. Therefore the single boundary condition in \cite{Camili2017} turns into a summation over all edges in our work.
		Despite those changes, the work in \cite{Camili2017} provides the foundation for the formulation of unique solutions to the tram network problem.
	\end{remark}

	\section{Numerical experiments}	
	\label{sec:numerik}
	We numerically solve the tram network using a finite difference method to describe the evolution of the passenger mass. For any edge $e \in \mathcal{E}$ we discretize the space interval $[0,l_e]$ using a stepsize $\Delta x^e>0$ such that $x_i^e = (i-1)\Delta x^e, ~i=1,\dots, K$. Similarly for the time component we introduce $\Delta t^e>0$ with $t_j^e = (j-1)\Delta t^e,~ j=1,\dots,J$. The initial data of the tram network problem in \eqref{eq:NetworkProblem} is discretized by 
	\begin{align*}
		\nu_i^{0,e} = \frac{1}{\Delta x^e}\int_{x^e_{i-\frac{1}{2}}}^{x^e_{i+\frac{1}{2}}} \nu_0^e(y)dy,
	\end{align*}
	where $x^e_{i+\frac{1}{2}} = (i-\frac{1}{2})\Delta x^e$.
	Boundary data is described by
	\begin{align*}
		\nu_0^{j,e} = \frac{1}{w^e} \int_{t_j}^{t_{j+1}} \tilde{\nu}^e_{x=0}(s)ds, \qquad \tilde{\nu}_{x=l_e}^e(t) = \begin{cases}
			w^e \nu_K^{j,e}, & ~t \in \lbrace{t_1,\dots,t_J}\rbrace\\
			0, & ~\text{else}.
		\end{cases}
	\end{align*}
	As in our problem information  exclusively travels from the left to the right, we apply a left-sided Upwind scheme
	\begin{align*}
		\nu_i^{j+1,e} = \nu_i^{j,e} - \frac{\Delta t^e}{\Delta x^e} w^e (\nu_i^{j,e} - \nu_{i-1}^{j,e}).
	\end{align*}
	For stability of the numerical method we require the CFL-condition \cite{Courant.1928} $\frac{\Delta t^e}{\Delta x^e} w^e \leq 1$. At the junctions we use the coupling conditions specified in the Equations \eqref{eq:coupling1-1}--\eqref{eq:coupling2-2}.
	
	For the passenger arrival we simulate the Poisson process via a thinning procedure \cite{Ogata.1981} that takes care of the time varying arrival rates, see Algorithm \ref{algo:thinning}. 
	
	\begin{algorithm}
		\caption{Simulation of an inhomogeneous Poisson process for the passenger arrivals via thinning}
		\begin{algorithmic}[1]
			
			\STATE \textbf{Input:} End time $T$, time-dependent rate function $\lambda(t)$, time points $t\_vec$ to evaluate $X(t)$
			\STATE \textbf{Output:} Arrival times $arrival\_times$, Poisson process $X(t\_vec)$
			
			\STATE Compute $\lambda_{\max} = \max_{t \in [0,T]} \lambda(t)$
			\STATE Initialize $t \gets 0$
			\STATE Initialize $arrival\_times \gets [\ ]$
			
			\WHILE{$t < T$}
			\STATE Draw $\Delta t \sim \text{Exponential}(\lambda_{\max})$
			\STATE $t \gets t + \Delta t$
			\IF{$t > T$}
			\STATE \textbf{break}
			\ENDIF
			\STATE Draw $u \sim \text{Uniform}(0,1)$
			\IF{$u \le \lambda(t)/\lambda_{\max}$}
			\STATE Append $t$ to $arrival\_times$
			\ENDIF
			\ENDWHILE
			
			\STATE \textbf{Construct Poisson process $X(t\_vec)$:}
			\FOR{each time point $t_i$ in $t\_vec$}
			\STATE $X(t_i) \gets$ number of arrivals in $arrival\_times$ that are $\le t_i$
			\ENDFOR
			
		\end{algorithmic}
		\caption{Thinning algorithm for the arrivals of passengers.}
		\label{algo:thinning}
	\end{algorithm}
	
	\paragraph{The data of the study}
	The following subsections are dedicated to simulate a real tram network and show its resilience under different restrictions. In the first part of the numerical study, we conducted simulation experiments on tram line 1 in Mannheim, which runs from north-west to the south-east and crosses the city center. Line 1 is considered to be the most important tram line in Mannheim and is the only one running throughout the night. A brief visualization of the northern part of the route of line 1 is shown in Figure \ref{fig:line1Schematic}. As we consider only one line, the entire network consists of 1-1 junctions.
	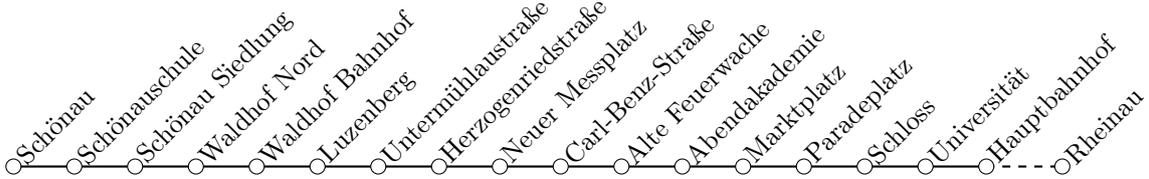
\begin{figure}[ht!]
		\centering
		\begin{tikzpicture}[
			station/.style={circle, draw, fill=white, inner sep=2pt},
			every node/.style={font=\small}
			]
			
			% Linie
			\draw[thick] (0,0) -- (12.8,0);
			\draw[dashed, thick] (12.8,0) -- (13.8,0);
			
			% Stationen
			\node[station] at (0,0) {};
			\node[below, rotate=45, anchor=west] at (0,0) {Schönau};
			
			\node[station] at (0.8,0) {};
			\node[below, rotate=45, anchor=west] at (0.8,0) {Schönauschule};
			
			\node[station] at (1.6,0) {};
			\node[below, rotate=45, anchor=west] at (1.6,0) {Schönau Siedlung};
			
			\node[station] at (2.4,0) {};
			\node[below, rotate=45, anchor=west] at (2.4,0) {Waldhof Nord};
			
			\node[station] at (3.2,0) {};
			\node[below, rotate=45, anchor=west] at (3.2,0) {Waldhof Bahnhof};
			
			\node[station] at (4.0,0) {};
			\node[below, rotate=45, anchor=west] at (4.0,0) {Luzenberg};
			
			\node[station] at (4.8,0) {};
			\node[below, rotate=45, anchor=west] at (4.8,0) {Untermühlaustraße};
			
			\node[station] at (5.6,0) {};
			\node[below, rotate=45, anchor=west] at (5.6,0) {Herzogenriedstraße};
			
			\node[station] at (6.4,0) {};
			\node[below, rotate=45, anchor=west] at (6.4,0) {Neuer Messplatz};
			
			\node[station] at (7.2,0) {};
			\node[below, rotate=45, anchor=west] at (7.2,0) {Carl-Benz-Straße};
			
			\node[station] at (8.0,0) {};
			\node[below, rotate=45, anchor=west] at (8.0,0) {Alte Feuerwache};
			
			\node[station] at (8.8,0) {};
			\node[below, rotate=45, anchor=west] at (8.8,0) {Abendakademie};
			
			\node[station] at (9.6,0) {};
			\node[below, rotate=45, anchor=west] at (9.6,0) {Marktplatz};
			
			\node[station] at (10.4,0) {};
			\node[below, rotate=45, anchor=west] at (10.4,0) {Paradeplatz};
			
			\node[station] at (11.2,0) {};
			\node[below, rotate=45, anchor=west] at (11.2,0) {Schloss};
			
			\node[station] at (12.0,0) {};
			\node[below, rotate=45, anchor=west] at (12.0,0) {Universität};
			
			\node[station] at (12.8,0) {};
			\node[below, rotate=45, anchor=west] at (12.8,-0.0) {{Hauptbahnhof}};
			
			\node[station] at (13.8,0) {};
			\node[below, rotate=45, anchor=west] at (13.8,-0.0) {{Rheinau}};
			
		\end{tikzpicture}
		\caption{Line 1 from Sch\"onau to Rheinau, with all stops up to the central station.}
		\label{fig:line1Schematic}
	\end{figure}
	
	The simulations represent a typical weekday (Monday to Friday). The tram passenger dynamics are modeled as before, stopping at each station, where passengers board and alight according to a Poisson processes. Boarding is limited by the tram’s seating and standing capacities, which vary hourly to reflect realistic occupancy patterns. We assume hourly arrival rates of passengers for each station and alighting shares also in hourly intervals for each station. 
	The data used is real data provided by the local tram company and is taken from 4 two week time intervals in the years from 2022 to 2024. Also the tram capacities are chosen as the real data depicts. 
	
	Table \ref{tab:DataOverview} shows an excerpt of the data that was used for our study and shows the boarding intensity (passengers per hour), alighting share (percentage of passenger in the tram alighting) and average seating and standing capacities of the tram (in number of passengers) in the hourly interval from 4 to 5~p.m. The two most important stations are Paradeplatz (the most popular square in the inner city) and the central station (Hauptbahnhof). Others like the university (Universität) or castle stop (Schloss) show lower numbers of boarding and alighting passengers due to their less central location. Also the tram capacity varies slightly as different types of trams with different space for passengers are used on line 1. 
	
	\begin{table}[h!]
		\centering
		\begin{tabular}{lcccc}
			\hline
			\textbf{Tram stop} & \makecell{\textbf{Boarding}\\\textbf{intensity}} & \makecell{\textbf{Seating}\\\textbf{capacity}} & \makecell{\textbf{Standing}\\\textbf{capacity}} & \makecell{\textbf{Alighting}\\\textbf{percentage}} \\
			\hline
		%	Alte Feuerwache (15h) & 127 & 117 & 138 & 12 \\
			Alte Feuerwache  & 109 & 113 & 135 & 12 \\
		%	Abendakademie (15h)   & 117 & 117 & 138 & 25 \\
			Abendakademie    & 112 & 113 & 135 & 24 \\
		%	Marktplatz (15h)      & 111 & 117 & 138 & 19 \\
			Marktplatz       & 119 & 113 & 135 & 20 \\
		%	Paradeplatz (15h)     & 157 & 117 & 138 & 27 \\
			Paradeplatz      & 142 & 114 & 136 & 26 \\
		%	Schloss (15h)         & 74  & 117 & 138 & 8  \\
			Schloss          & 66  & 113 & 135 & 7  \\
		%	Universität (15h)     & 39  & 112 & 136 & 2  \\
			Universität      & 41  & 115 & 136 & 2  \\
		%	Hauptbahnhof (15h)    & 167 & 112 & 136 & 36 \\
			Hauptbahnhof     & 169 & 115 & 136 & 36 \\
			\hline
		\end{tabular}
		\caption{Data of some tram stops on line 1 (rounded on integers) for the time from 4 to 5~p.m..}
		\label{tab:DataOverview}
	\end{table}
	
	\paragraph{Performance measures of the tram network}
	In order to be able to compare the performance of different scenarios in the tram network, we introduce two measures that quantify the performance. 
	One quantity of interest that we consider are waiting times at the tram stops and observe the consequences on passenger for changed tram service frequencies. The total waiting time in the tram system is defined by
	\begin{align*}
		\sum_{e \in \mathcal{E}} \int_0^T q^e(t)dt.
	\end{align*}
%	The total waiting time is influenced by the frequency of the tram service and also the circumstand whether a tram is already full and not all passengers are able to board the next tram.
	As a second quantity of interest measuring a certain service level of passenger we consider how much time is spent by passengers in trams without an available seat. Generally trams have a seat capacity, given by $\tau_\text{seat}^e$ and a second capacity of passengers standing in the tram. Together they make up the entire tram capacity $\tau^e$. We assume that passengers take a free seat if there is one available. If there is no free seat in the tram they have to stand. The total standing time is defined by
	\begin{align*}
		\int_0^T	\sum_{e \in \mathcal{E}} \int_0^{l_e} \max\left\lbrace0,\nu^e(x,t) - \tau_\text{seat}^e\left(t-\frac{l_e}{w^e}\right)\right\rbrace dx dt.
	\end{align*}
Both performance measure can also be considered just on a subset of the entire tram stops, which replaces the sum over all edges to a sum just over the relevant stations. Those and similar performance measures have been considered in literature, e.g.\ in \cite{NARAYAN2017}.

To account for uncertainty from the passenger arrival process and, later, from tram service disruptions, we perform a Monte Carlo simulation with 1000 runs throughout the numerical analysis to estimate the average total waiting and standing times.

	\subsection{Different tram frequencies}
	\label{sec:numerik_Frequencies}
	The primary objective of the first experiment is to analyze passenger flow dynamics along line 1, for varying the tram service frequency. The current state of the art is a 10-minute interval between 5~and 9~p.m. During off-peak hours in the evening, trams run at 20-minute intervals. Throughout the rest of the night the trams are scheduled hourly.
	We investigate alternative service intervals during the main period (5~a.m.\ to 9~p.m.), considering 5-, 20-, 30-, and 40-minute intervals. In the other time periods the schedule remains unchanged.  %We will observe the second aspect in a more than linear increase of the waiting time if the tram frequency is low.

%	\begin{align*}
%		\int_0^T	\sum_{e \in \mathcal{E}} \int_0^{l_e} \nu^e(x,t) - \tau^e(t-\frac{(l_e)^2}{\int_0^{x}w^e(y)dy}) dx dt.
%	\end{align*}
	
%	In Figure \ref{fig:linie1_waiting_station}, passenger waiting times are analyzed for different tram service frequencies from two perspectives: spatial, aggregated by station, and temporal, aggregated by hourly intervals.
	
	The spatial analysis (left part of Figure \ref{fig:linie1_waiting_station}) highlights stations with higher passenger demand, as indicated by longer average waiting times. This effect is particularly pronounced at stations such as Sch\"onau, Waldhof Bahnhof, Hauptbahnhof, and Tattersall, where a greater number of passengers arrive and may need to wait for subsequent trams. Across the five scenarios, the highest service frequency (one tram every five minutes) consistently yields the shortest waiting times. Waiting times increase moderately for 10-, 20-, and 30-minute intervals, and escalate sharply under the 40-minute interval scenario. 
	
	This nonlinear increase can be attributed to insufficient tram capacity at low service frequencies, resulting in a substantial proportion of passengers waiting for a second or even third tram. In the present example, this phenomenon is most evident when trams enter and exit the city center (around Alte Feuerwache and after Tattersall), which are the sections of line 1 experiencing the highest passenger loads.
	\begin{figure}[ht!]
	\begin{minipage}{0.49\textwidth}
	\centering
	\includegraphics[width=1\linewidth]{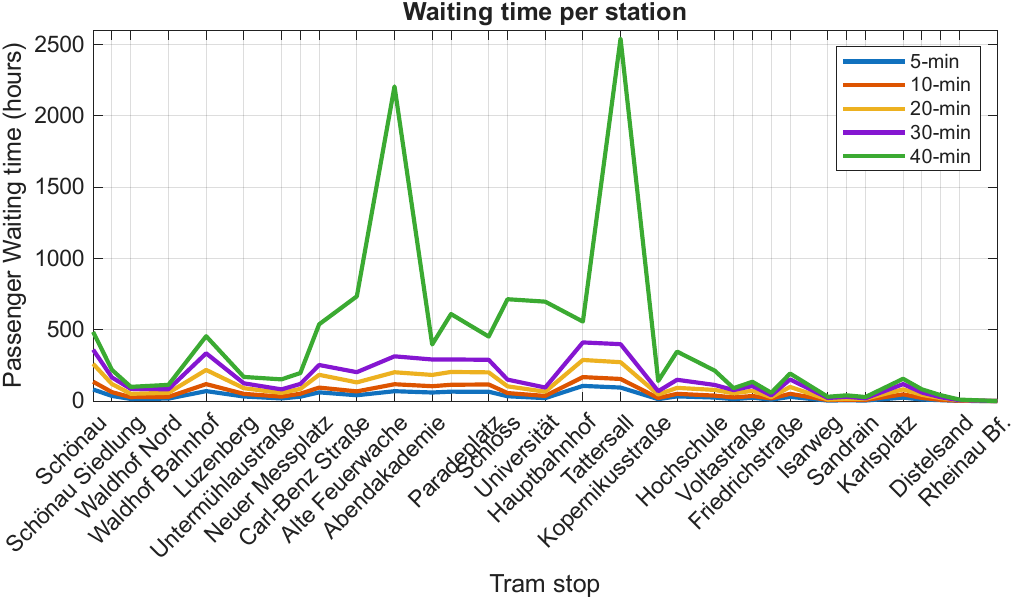}
	\end{minipage} ~\begin{minipage}{0.49\textwidth}
	\centering
	\centering
	\includegraphics[width=1\linewidth]{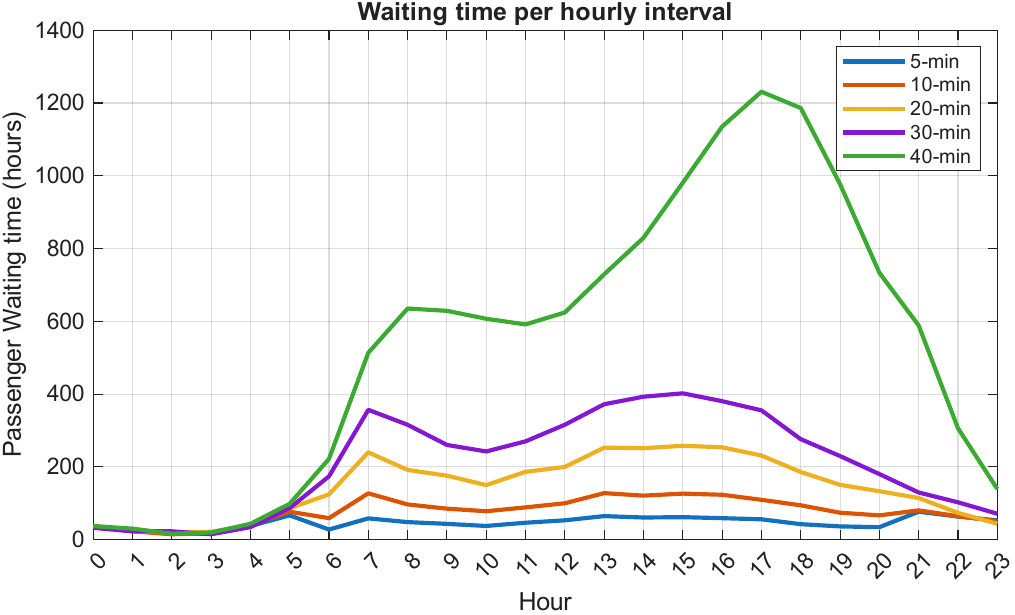}
	%\captionof{figure}{The passenger waiting times (in hours) per hourly interval over all tram station of line 1 for different tram frequencies.}
	%\label{fig:linie1_waiting_hour}
	\end{minipage}
	\caption{The passenger waiting times (in hours) per tram station (left) and hourly interval (right) for different tram frequencies.}
	\label{fig:linie1_waiting_station}
\end{figure}

Considering the temporal component of waiting times in the right part of Figure \ref{fig:linie1_waiting_station}, we observe a similar pattern: low waiting times for high service frequencies and higher waiting times for low frequencies. In the early morning and late evening, when tram services are the same across all scenarios, waiting times are identical. Peak periods can be observed in the morning and afternoon. For the 30-minute frequency, and especially the 40-minute frequency, waiting times extend beyond the standard rush hours, particularly in the evening, as tram capacity is insufficient to transport all passengers to their destinations.

A final overview of total waiting times across the tram network, depending on the chosen service frequency, is presented in Figure \ref{fig:linie1_waiting_total}. The figure clearly shows the drastic and nonlinear increase in waiting times for the 40-minute frequency.

	\begin{figure}[ht!]
		\centering
		\includegraphics[width=0.8\linewidth]{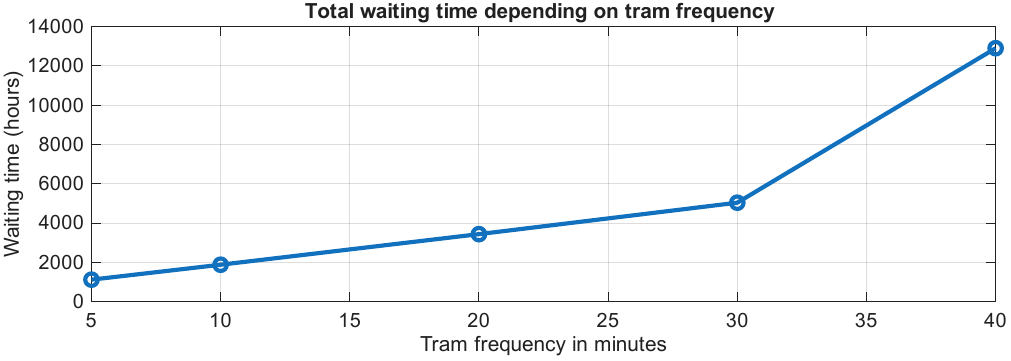}
		\caption{The total passenger waiting time (in hours) for different tram frequencies.}
		\label{fig:linie1_waiting_total}
	\end{figure}

	Next, we focus on the second quantity of interest: the aggregated time passengers have to remain standing in the trams, as shown in Figures \ref{fig:linie1_standing_station}--\ref{fig:linie1_standing_total}. 
	
	The first analysis considers standing passengers from a spatial perspective (left part of Figure \ref{fig:linie1_standing_station}). For the 5- and 10-minute service frequency, trams provide sufficient capacity for each passenger to have a seat. At a 20-minute frequency, however, passengers in the city center and the stations immediately before and after may be unable to find a seat. The number of standing passengers increases further at lower frequencies, despite there being fewer trams available in which passengers could stand.
	\begin{figure}
	\begin{minipage}{0.49\textwidth}
		\centering
		\includegraphics[width=1\linewidth]{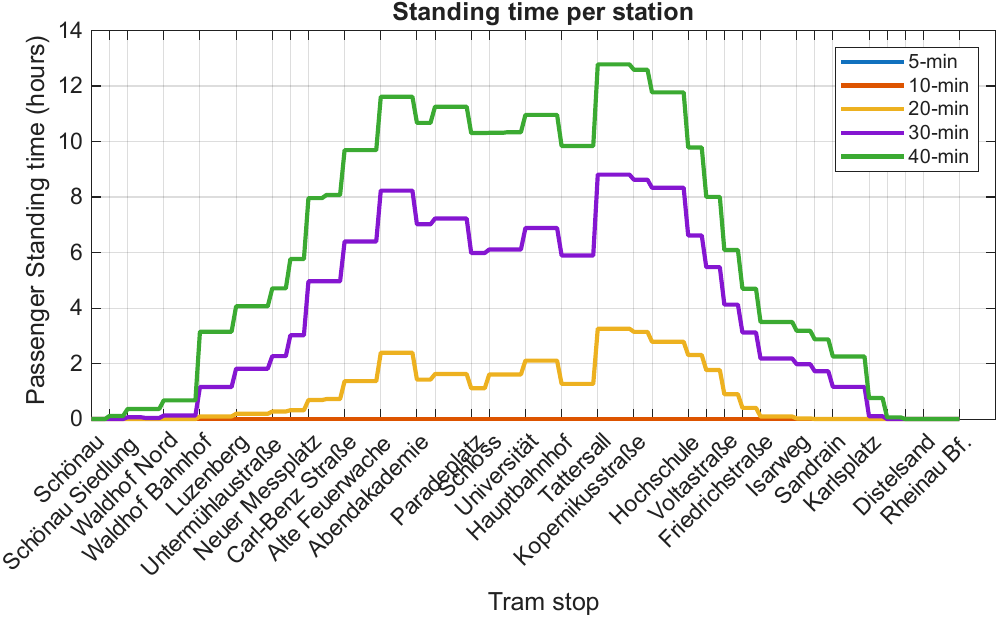}
	\end{minipage} ~\begin{minipage}{0.49\textwidth}
		\includegraphics[width=1\linewidth]{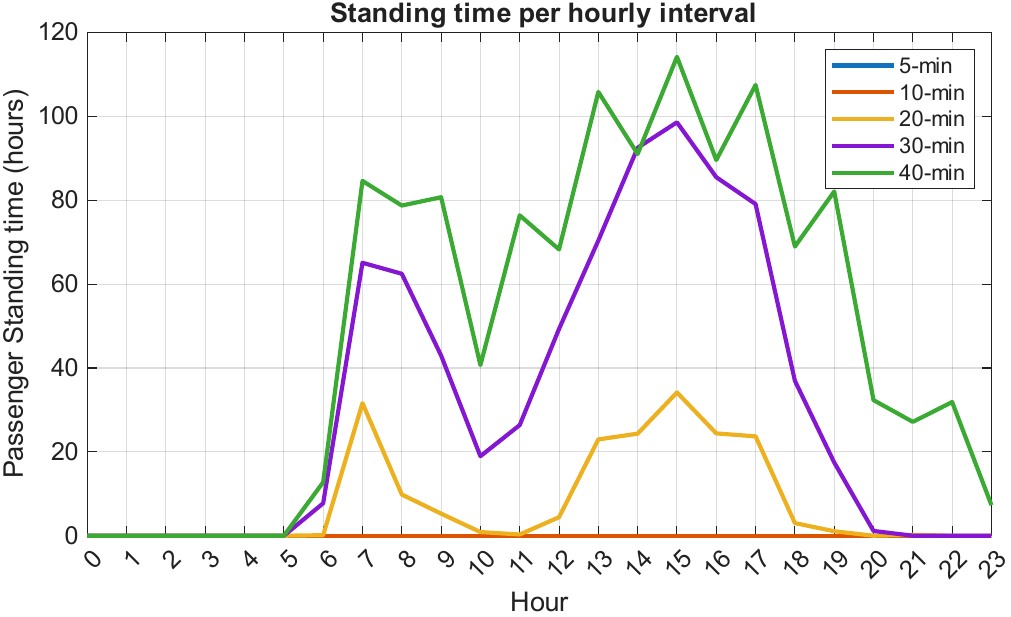}
	%	\captionof{figure}{The passenger standing times (in hours) per hourly interval over the entire route of line 1 for different tram frequencies.}
	%	\label{fig:linie1_standing_hour}
	\end{minipage}
	\caption{The passenger standing times (in hours) for all stations (left) and for hourly intervals (right) for different tram frequencies.}
	\label{fig:linie1_standing_station}
\end{figure}
	
	From a temporal perspective (right part of Figure \ref{fig:linie1_standing_station}), rush hours are again clearly visible in the morning and the afternoon. Interestingly, during the main rush hours, the standing times for the 30- and 40-minute frequencies are quite similar. This can be explained by the physical limit of standing space within the trams, which constrains the number of passengers that can board when service frequency is low. Consequently, more passengers are unable to board immediately and must wait at the tram stop and therefore rather increase the waiting times. This effect is also evident in Figure \ref{fig:linie1_standing_total}, where the increase in total standing time for the 40-minute frequency is relatively moderate and not that steeply increasing as it was for the waiting time.
	
	\begin{figure}[ht!]
		\centering
		\includegraphics[width=0.8\linewidth]{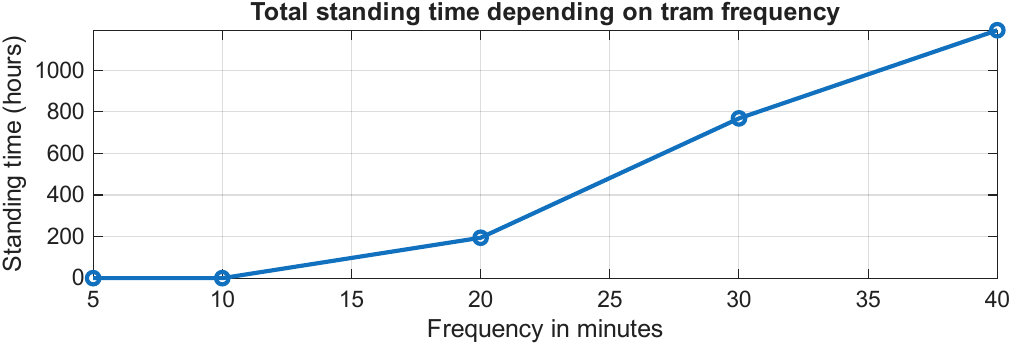}
		\caption{The total passenger standing times (in hours) for different tram frequencies.}
		\label{fig:linie1_standing_total}
	\end{figure}

	\subsection{Delays in the passenger exchange}
	\label{sec:numerik_Delays}
In this subsection, we consider an additional factor that has been neglected so far but can influence the tram network. We extend our standard model by allowing trams to stop at stations and continue their journey after a dwell time $\pi^e$, which may depend on the time $t$ as well as the number of boarding and alighting passengers. This requires to adjust the admissibility condition of the tram schedule in Definition \ref{def: admissibleSchedule}.
For this numerical experiment, the family of functions $\tau = (\tau^e){e \in \mathcal{E}}$, with $\tau^e: [0,\infty) \rightarrow \mathbb{R}{\geq 0}$ representing the admissible tram schedule, is modified in aspect b) to:
\begin{itemize}
%			\item[a)] as before%For any $e \in \mathcal{E}$, $\tau^e(t)<\infty$ for all $t \in [0,\infty)$
	\item[b)] For any $t\in[0,\infty)$ and $v\in\mathcal{V}\backslash\mathcal{V}^\text{terminal}$, there exists a time parameterized \textit{injective except for the empty set} map $\delta_t^v: v^- \rightarrow \lbrace v^+ \bigcup \emptyset\rbrace$ such that for any $e \in v^-$ with $\tau^e\left(t-\frac{l_e}{w^e}- \pi^e(t, b^e(t),a^e(t))\right) >0$ the function $\delta_t^v$ maps to exactly one of the outgoing edges from the set $v^+$. Furthermore it holds 
	\begin{align*}
		\tau^{\delta_t^v(e)}(t) = \tau^e\left(t-\frac{l_e}{w^e}- \pi^e(t, b^e(t),a^e(t))\right).
	\end{align*}
	For any $e \in v^-$ with $\tau^e\left(t-\frac{l_e}{w^e}- \pi^e(t, b^e(t),a^e(t))\right)=0$, we set $\delta_t^v(e) = \emptyset$. If $v \in \mathcal{V}^\text{terminal}$ then for any $e \in v^-$ with $\tau^e\left(t-\frac{l_e}{w^e}- \pi^e(t, b^e(t),a^e(t))\right)>0$ the function $\delta_t^v$ may also map to $\emptyset$ (terminating tram; particularly $v^+$ may be empty). 
%			\item[c)] as before
\end{itemize}
	We assume that the delay function for the passenger exchange is the same for all stations, is not dependent on the time explicitly and for any $e \in \mathcal{E}$ has the form
	\begin{align*}
		\pi^e (t,b^e(t), a^e(t)) = \begin{cases}
			0, & b^e(t)+a^e(t)\leq 50,\\
			\frac{1}{50}(b^e(t)-a^e(t)-50), & b^e(t)+a^e(t)> 50.
		\end{cases}
	\end{align*}
	In general, the patterns of waiting and standing times presented in the previous section remain similar when accounting for additional delays due to passenger exchange, although the times increase slightly. Therefore, we focus on the direct impact of these delays.
	
	Figure \ref{fig:linie1_delay_space} illustrates the evolution of average delays along line 1 during the course of a journey. At high service frequencies (5- and 10-minute intervals), the model predicts virtually no delays caused by passenger exchange, as trams are not crowded due to their high frequency. At lower frequencies, however, average delays increase. Most of the delay accumulates in the city center (between Alte Feuerwache and Tattersall), indicating that the total passenger exchange is highest in this section. Unsurprisingly, the largest increase in delays occurs at Mannheim Hauptbahnhof (main station). These irregularities in delay propagation also contribute to increased overall waiting times for passengers at the tram stops.
	\begin{figure}[ht!]
		\centering
		\begin{minipage}{0.54\linewidth}
		\centering
		\includegraphics[width=1\linewidth]{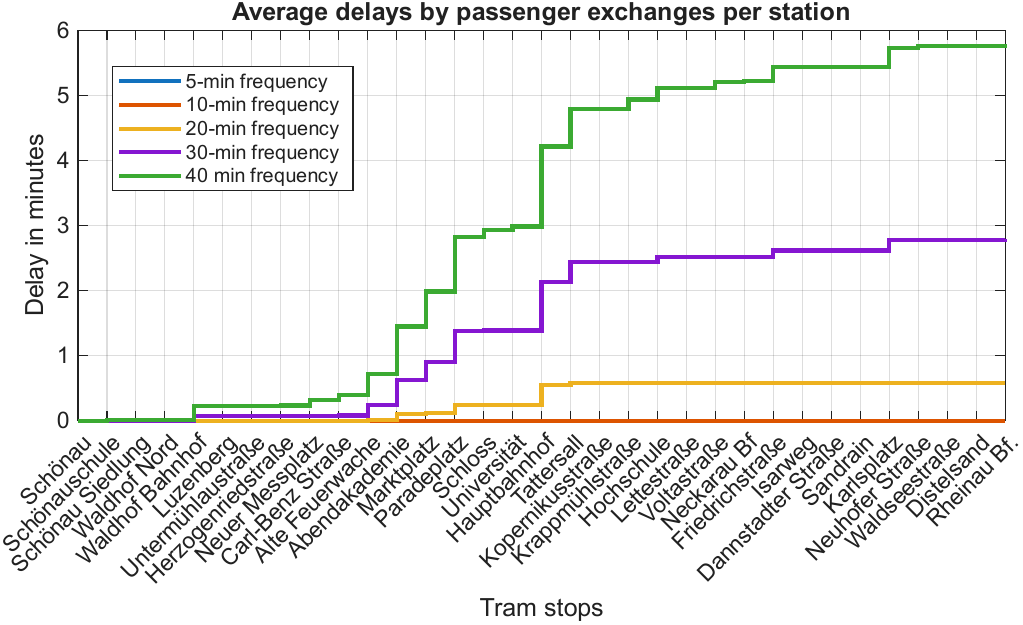}
		\captionof{figure}{The average delays in space caused by the longer passenger exchange times for the different tram frequencies.}
		\label{fig:linie1_delay_space}
	\end{minipage} ~\begin{minipage}{0.44\linewidth}
	\centering
	\includegraphics[width=1\linewidth]{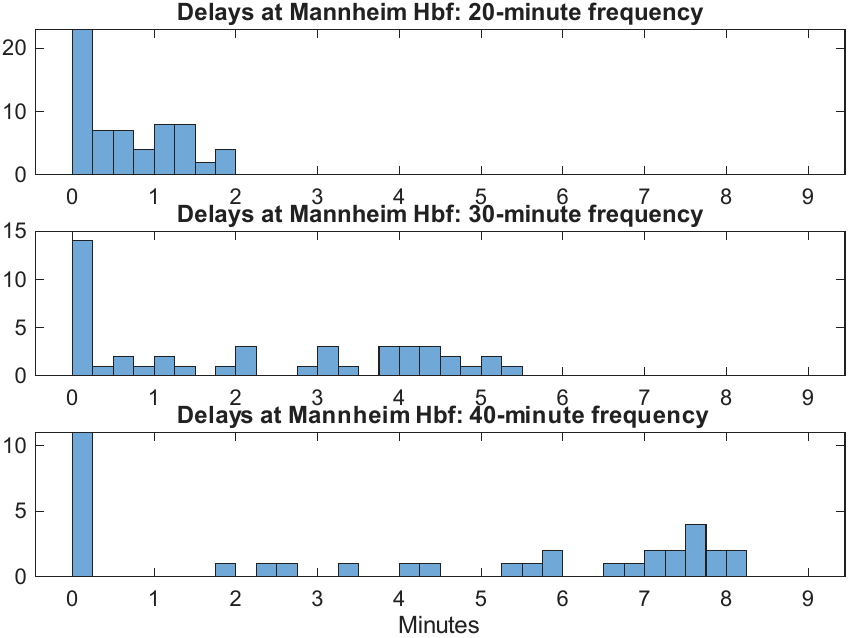}
	\captionof{figure}{Histograms of the tram delays at Mannheim Hauptbahnhof for different tram frequencies.}
	\label{fig:linie1_delay_totalMaHbf}
	\end{minipage}
	\end{figure}
%	\begin{figure}[ht!]
%		\centering
%		\includegraphics[width=0.75\linewidth]{Figures_Strassenbahn/linie1_delay_space.eps}
%		\caption{The average delays in space caused by the longer passenger exchange times for the different tram frequencies.}
%		\label{fig:linie1_delay_space}
%	\end{figure}

	In Figure \ref{fig:linie1_delay_totalMaHbf}, we examine the distribution of delays at the Mannheim Hauptbahnhof tram stop in more detail and present a histogram of the cumulative delays observed upon arrival at Hauptbahnhof. Only the 20-, 30-, and 40-minute service frequencies are shown, as only negligible delays occur at the 5- and 10-minute frequencies in our setting. For the 20-minute frequency, apart from the services that run on time, delays are approximately uniformly distributed up to 2 minutes. These delays likely correspond to services during the morning and afternoon rush hours, when passenger loads are highest. At the 30-minute frequency, the distribution is broader and includes delays exceeding 5 minutes. There appears to be a clustering of delays around 4–5 minutes, with fewer services experiencing delays of 1–3 minutes. At the 40-minute frequency, this clustering shifts further, with most delays occurring around 7–8 minutes. Only a few services remain on schedule, primarily those operating in the early morning.
	%\textcolor{red}{Die Figure mit den am Hbf entstandenden Delays habe ich rausgenommen}

Last, we present a trajectory plot for the scenario with delays, additionally incorporating random tram failures. We assume that 0.5\% of stops experience a breakdown of 8~minutes, and a further 1\% of trams experience a 4~minute breakdown. These additional delays are incorporated in the same manner as passenger-exchange delays.

Figure~\ref{fig:linie1_delay_trajectories10_20} shows trajectory plots of line~1 during the morning rush hours, between 5 and 9~a.m., with a 20-minute (left) and 10-minute interval (right). In particular, in the area before the city center around Alte~Feuerwache, and at approximately 7--8~a.m., we observe highly crowded trams for the lower frequency, indicated by orange or red colors in the trajectories corresponding to highly crowded trams, whose seating and standing capacity is almost reached.

Steeper slopes in the trajectories correspond to delays, either due to passenger exchange or tram failures. Focusing on the tram departing from Sch\"onau at 8:07, a first minor delay occurs at Waldhof Nord, followed by a larger delay at Abendakademie, caused by a tram failure and high passenger exchange. As a result of these delays, the tram carries a large number of passengers. Furthermore, the interval to the preceding tram is considerably longer than 20 minutes, whereas the interval to the following tram is significantly shorter. Consequently, this tram collects passengers who would otherwise have traveled on the next tram, making the subsequent tram less crowded.

For a 10-minute service frequency, in general, tram occupancy is lower, while spatial and temporal patterns of higher passenger loads remain comparable to those observed in the 20-minute frequency scenario. Although service interruptions occur, their effects on the overall network remain more limited.
Still, yellow color levels indicate trams in which seating capacity is fully exhausted. From the passengers point of view, a higher utilization level is undesirable and therefore also in the 10-minute frequency scenario some at full seating capacity. 
\begin{figure}[ht!]	
	\centering
	\begin{minipage}{0.49\linewidth}
	\centering
	\includegraphics[width=1\linewidth]{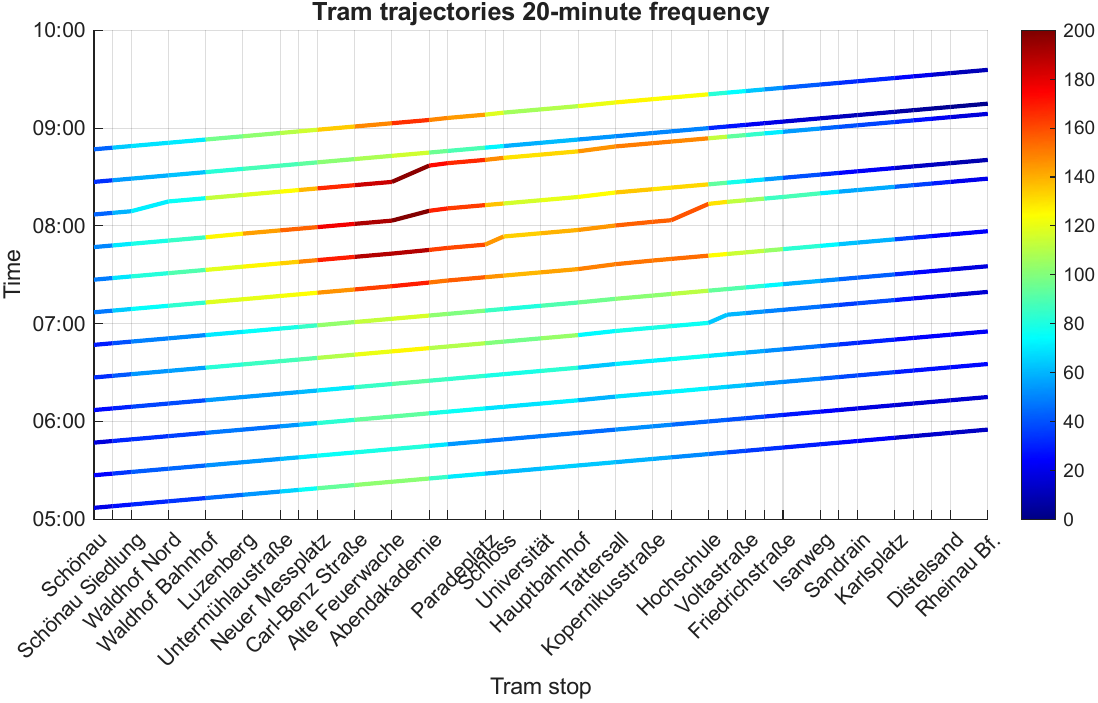}
	\end{minipage} ~\begin{minipage}{0.49\linewidth}
			\centering
		\includegraphics[width=1\linewidth]{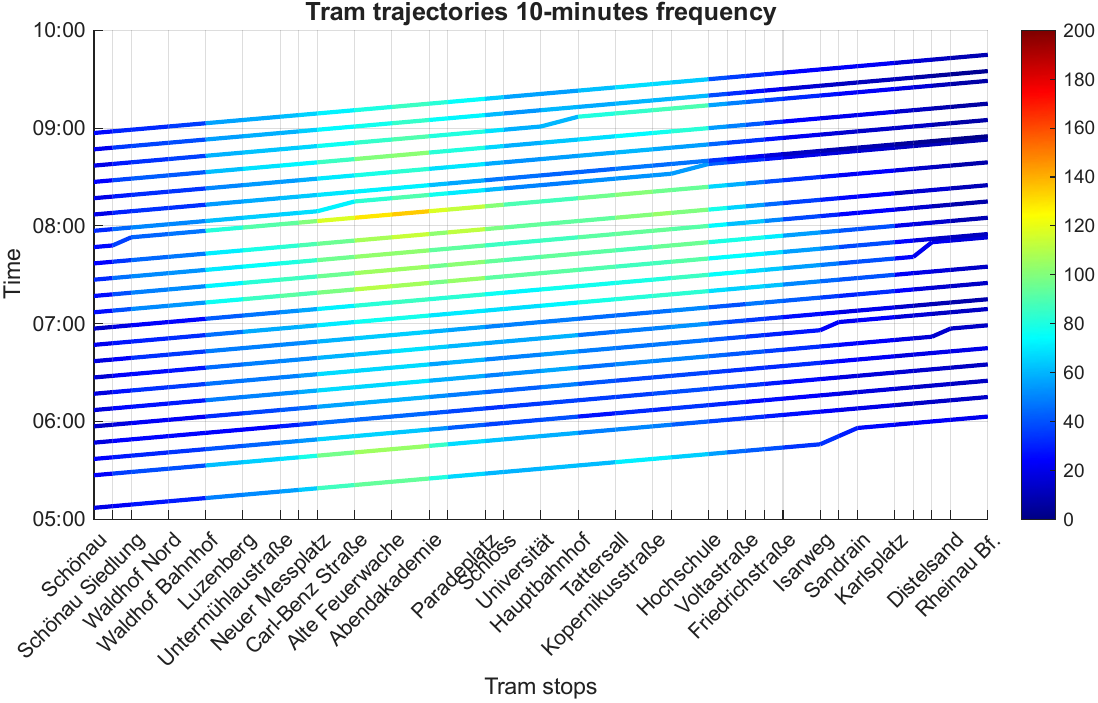}
	\end{minipage}
	\caption{Trajectories of the trams in the morning with a 20-minute frequency (left) and 10-minute frequency (right) of tram service. The coloring shows the number of passengers in the tram.}
	\label{fig:linie1_delay_trajectories10_20}
\end{figure}

%	\textcolor{red}{Generell alles mit Monte Carlo!! an den Anfang um dopplungen zu vermeiden.}
	
	\subsection{Cancellation of tram services}
	\label{sec:numerik_Canellations}
	An even more severe disruption in the tram network is the cancellation of tram services, whose impact on waiting and standing times is analyzed in this subsection.
	Tram service cancellations may occur due to spontaneous technical failures, track obstructions, or staff shortages. In some cases, entire vehicle circulations are canceled for a full day, whereas in other cases only individual trips are affected and can be reinstated quickly. Both types of disruptions are represented in our model by distributing half of all canceled trips evenly over the course of the day, while the remaining half is drawn uniformly from all scheduled trips.
%	To obtain a comprehensive picture of different disruption scenarios, we conduct a Monte Carlo simulation in which a predefined percentage of trips is cancelled in each run.
	The cancellation of individual trips can have markedly different effects on the performance measures for waiting and standing times. For instance, the cancellation of the last trip of the day results, in our model, in passengers having to wait until the first tram service on the following morning.%, therefore the results are based on a Monte Carlo simulation averaging over 1{,}000 runs.
	
	Figure~\ref{fig:cancelWaiting} (left) illustrates how increasing cancellation rates affect the average waiting time in the system.
	For each tram frequency, the shaded areas indicate the 20th and 80th percentiles around the mean.
	For all frequencies, we observe an increase in waiting times as the cancellation rate rises.
	Moreover, the spread between the 20th and 80th percentiles widens with increasing cancellation rates.
	While the increase is almost linear for higher service frequencies, the 30-minute frequency exhibits a pronounced deterioration for cancellation rates exceeding 20\%.
	In this regime, the system begins to collapse, resulting in a steep increase in waiting times. For these scenarios, there are periods during the day in which tram capacity is no longer sufficient, preventing a substantial number of passengers from boarding the vehicles.
	
	Interestingly, a cancellation rate of 30\% under a 10-minute frequency leads to average waiting times comparable to those observed for a 20-minute frequency without any cancellations, despite the higher number of operating trams in the former case.
	This effect arises because canceled services are unevenly distributed over time, increasing the likelihood of extended gaps between consecutive trams and thereby significantly increasing total waiting times.

	\begin{figure}[ht!]
		\centering
		\begin{minipage}{0.49\linewidth}
			\centering
			\includegraphics[width=1\linewidth]{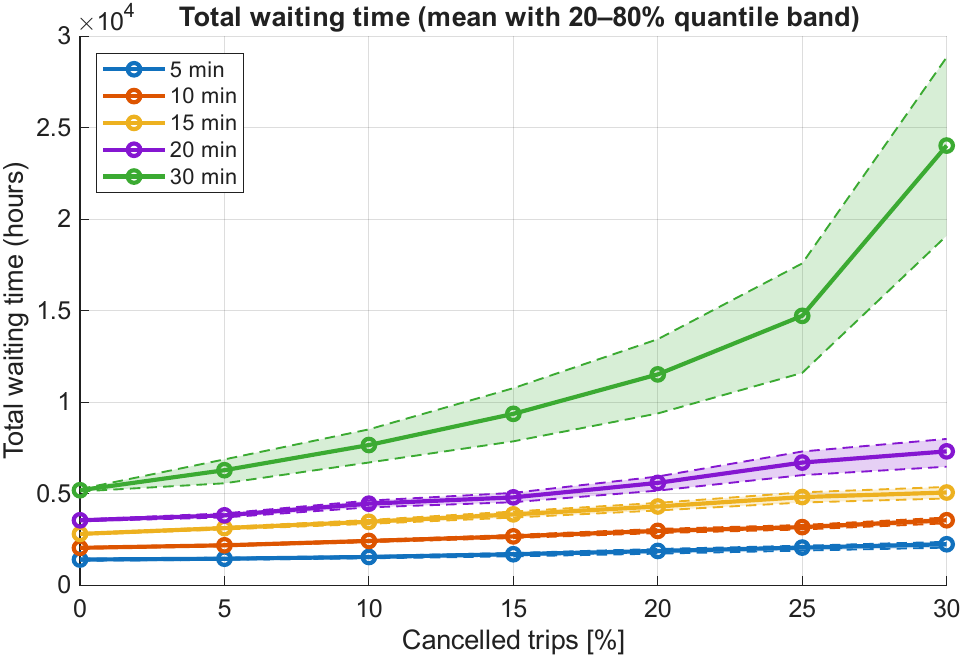}
			
		\end{minipage} ~\begin{minipage}{0.49\linewidth}
			\centering
				\includegraphics[width=1\linewidth]{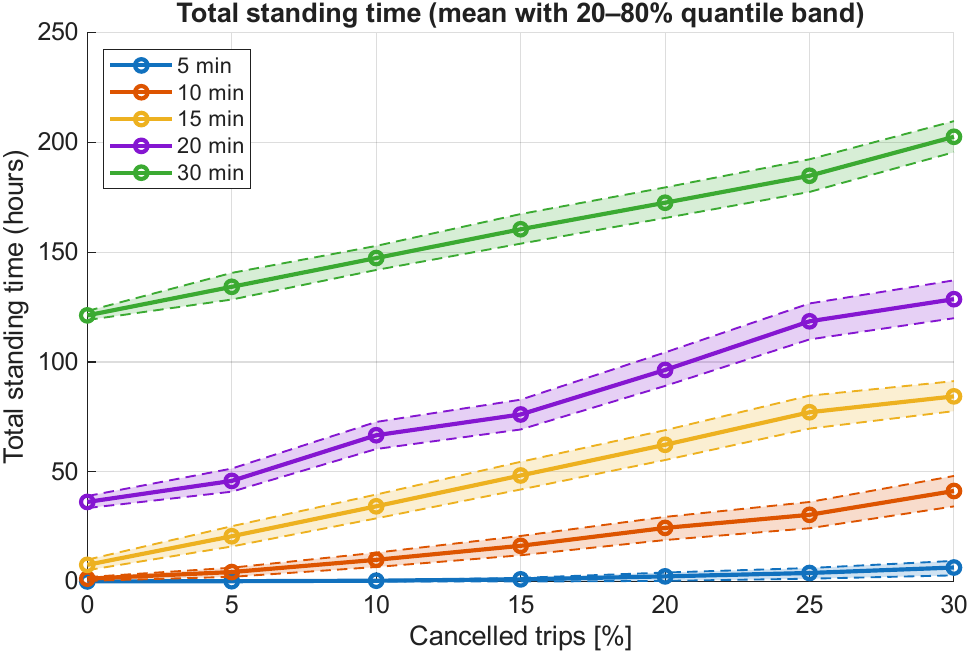}
		\end{minipage}
		\caption{The average total passenger waiting (left) and standing (right) time depending on the cancellation probability for different frequencies.}
		\label{fig:cancelWaiting}
	\end{figure}
	A similar analysis is done in the right part of Figure \ref{fig:cancelWaiting}, where we show the evolution of the standing time for different frequencies and cancellations rates. For all presented frequency scenarios we observe linear increases in the standing time without significant outliers for the 30-min frequency. This can be explained by the fact, that trams are already full at some point in time and any additional passenger must wait at the tram stop and do not influence the number of standing passengers.
	
\subsection{Simulation of a tram network}
\label{sec:numerik_Network}
As a final numerical investigation, we extend the analysis from a single tram line to an entire network containing general $m$--$n$ junctions.
To this end, we consider a section of the northern Mannheim tram network centered around the tram stop Alte Feuerwache.
Five different tram lines reach Alte Feuerwache from different directions.
A schematic overview of the network is shown in Figure~\ref{fig:TramNetwork}.

We adopt a directed graph representation and focus on tram services operating in the direction of Paradeplatz.
The nodes shown represent all tram stops that either constitute starting points of a line or correspond to junctions other than simple 1--1 connections, namely 2--1 junctions (Luzenberg and Bonifatiuskirche) and a 3--2 junction (Alte Feuerwache).
Between the depicted nodes, additional sequences of 1--1 junctions exist, which are not visualized in the figure.
The analysis is carried out up to the stop Paradeplatz, which represents the most central tram stop in Mannheim’s inner city.
 %The lines 1,2,3 and 4 run on a regular 10-minute frequency from 6:00am to 9:00~p.m., and a 20 minutes frequency before and after. Alle tram lines except for line 1 stop their service around midnight, while line 1 continus on a hourly basis. Line 15 only runs in the rush hours from 7:00-9:00am and 2:00~p.m.-7:00~p.m. on a 20-minute frequency.

\tikzset{
	stop/.style={
		ellipse,
		draw,
		minimum width=20mm,
		minimum height=8mm,
		inner sep=0pt,
		align=center
	}
}

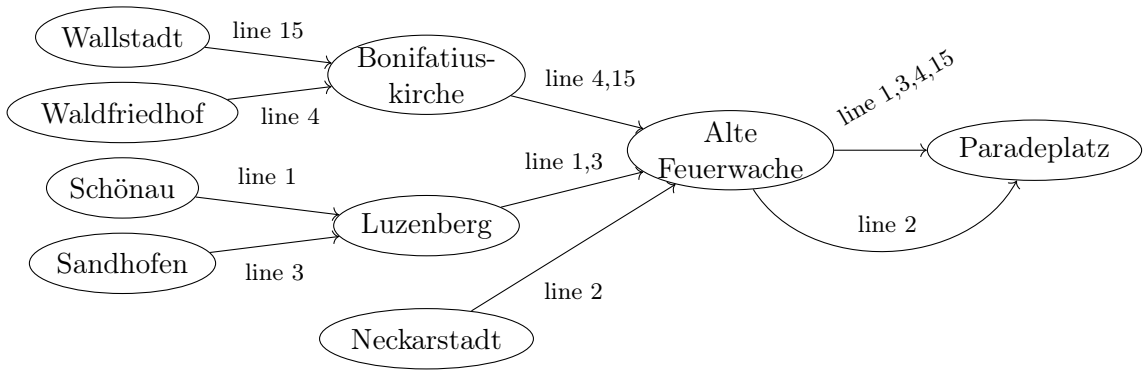
\begin{figure}[htbp]
	\centering
	\begin{tikzpicture}[
		node distance=1.5cm and 2cm,
		every node/.style={circle, draw, minimum size=6mm},
		line/.style={thick}
		]
		
		% Startknoten (links)
		\node[stop] (A1) at (0,2) {Wallstadt};
		\node[stop] (A2) at (0,1) {Waldfriedhof};
		\node[stop] (B1) at (0,0) {Schönau};
		\node[stop] (B2) at (0,-1) {Sandhofen};
		\node[stop] (B3) at (4,-2) {Neckarstadt};
		
		% Vereinigungsknoten
		\node[stop] (C1) at (4,1.5) {Bonifatius-\\kirche};
		\node[stop] (C2) at (4,-0.5) {Luzenberg};
		
		% Endknoten
		\node[stop] (D) at (8,0.5) {Alte\\Feuerwache};
		\node[stop] (E) at (12,0.5) {Paradeplatz};
		
		% Verbindungen oben
		\draw[->] (A1) -- node[midway, above, draw=none, fill=none, inner sep=-4pt] {\footnotesize line 15} (C1);
		\draw[->] (A2) -- node[midway, below, draw=none, fill=none, inner sep=-9pt, xshift=4pt] {\footnotesize line 4} (C1);
		
		% Verbindungen unten
		\draw[->] (B1) -- node[midway, above, draw=none, fill=none, yshift=-5pt] {\footnotesize line 1} (C2);
		\draw[->] (B2) -- node[midway, below, draw=none, fill=none, yshift=6pt] {\footnotesize line 3} (C2);
		
		% Zusammenführung
		\draw[->] (C1) -- node[midway, above, draw=none, fill=none, yshift=-10pt, xshift=5pt] {\footnotesize line 4,15} (D);
		\draw[->] (C2) -- node[midway, above, draw=none, fill=none, yshift=-10pt, xshift=-3pt] {\footnotesize line 1,3} (D);
		\draw[->] (B3) -- node[midway, below, draw=none, fill=none] {\footnotesize line 2} (D);
		
		\draw[->] (D) -- node[midway, above, draw=none, fill=none, rotate=30, xshift=17pt, yshift=-12pt]
		{\footnotesize line 1,3,4,15} (E);
		\draw[->] (D) to[bend right=60] node[midway, above, draw=none, fill=none, yshift=-6pt]
		{\footnotesize line 2} (E);
		
	\end{tikzpicture}
	
	\caption{Schematic overview of the considered tram network with its center Alte Feuerwache.}
	\label{fig:TramNetwork}
\end{figure} 
In Table~\ref{tab:tramservices}, we provide the most important schedule information for the tram stop Alte Feuerwache. Lines 1--4 operate throughout the entire day, while line~15 serves as an additional line during the main rush hours. The last column of Table~\ref{tab:tramservices} lists the planned departure times at Alte Feuerwache in \texttt{hh:mm} format, with particular emphasis on the last digit of the minute, which plays a crucial role under a 10-minute service frequency. It can be observed that lines~3 and~15 share the same scheduled departure time. In our model, this violates the definition of an admissible tram schedule (see Definition~\ref{def: admissibleSchedule}). Consequently, the departures of line~15 between Alte~Feuerwache and Paradeplatz have been shifted by one discretization time step~$\Delta t$.

\begin{table}[ht!]
	\centering
	\begin{tabular}{r||c|c||c|c||c}
		\hline
		line & \makecell{peak\\ frequency} & \makecell{peak\\ service hours} & \makecell{off peak\\ frequency} & \makecell{off peak\\ service hours}& \makecell{departure\\ minute} \\
		\hline
		\hline			
		1 & 10-min & 5~a.m. --  9~p.m. & \makecell{20-min\\60-min} & \makecell{9~a.m.--midnight \\ midnight--5~a.m.} & xx:x3 \\
		\hline			
		2 & 10-min & 6~a.m. -- 9~p.m.& 20-min & 9~p.m. -- midnight & xx:x8 \\
		\hline			
		3 & 10-min & 6~a.m. -- 9~p.m.& 20-min & 9~p.m. -- 1~a.m. & xx:x9\\
		\hline			
		4 & 10-min & 7~a.m. -- 9~p.m.& 20-min & 4--7~a.m., 9~p.m.-1~a.m. & xx:x0\\
		\hline			
		15 & 20-min  & 7--9~a.m., 2--7~p.m. &-&-& xx:x9
	\end{tabular}
	\caption{Tram line information at the stop Alte Feuerwache according to the schedule.}
	\label{tab:tramservices}
\end{table}
We note that on certain sections of the network additional tram lines are in operation (e.g.\ in the vicinity of Bonifatiuskirche). However, in the following analysis we restrict our attention to tram lines that stop at Alte~Feuerwache and consider only passengers traveling on these lines.

To illustrate the basic structural properties of the network, we assume in this subsection that there are no delays or tram cancellations, although these effects could be incorporated without difficulty. %In order to account for uncertainties in passenger arrivals, we employ a Monte Carlo simulation with 1{,}000 runs.

%\paragraph{Capacity utilization}
The performance measure considered in this subsection is the capacity utilization $cu$ of the trams, defined as the ratio of the number of passengers to the number of available seats,
\begin{align*}
	cu = \frac{\text{number of passengers}}{\text{number of available seats}}.
\end{align*}
A utilization value exceeding~1 indicates that passengers are required to stand. Since all tram lines under consideration operate towards the stop Alte~Feuerwache, capacity utilization is measured at this location.

Figure~\ref{fig:CapacityUtilization} displays the capacity utilization of lines~1,~2,~3,~4, and~15 over the course of the day. For clarity of visualization, tram services after midnight are omitted. The results exhibit the characteristic bimodal demand pattern commonly observed in public transportation systems: a morning peak around 8~a.m., corresponding to commuting to work or school, and a second peak in the afternoon around 4~p.m., when passengers return.
On the left of Figure~\ref{fig:CapacityUtilization} we show the capacity utilization according to the true schedule.
Among the considered lines, line~3 shows the highest capacity utilization. This can be attributed to two factors. First, the section between Luzenberg and Paradeplatz exhibits particularly high passenger demand. Second, in the time window between \texttt{xx:x3} and \texttt{xx:x9} (a six-minute interval), no other tram line operates directly towards Paradeplatz. Consequently, all waiting passengers during this interval board line~3. Shortly thereafter, line~15 (\texttt{xx:x9}~+~$\Delta t$), line~4 (\texttt{xx:x0}), and line~1 (\texttt{xx:x3}) follow at shorter headways. Line~2 takes a detour and is therefore not selected by the majority of passengers.

For all lines except line~3, the capacity utilization remains below~1, indicating that under undisturbed operating conditions the available seating capacity is sufficient. In contrast, line~3 exhibits standing passengers during the intervals from 7--9~a.m. and 2--5~p.m. Note that a utilization value below~1 does not necessarily imply that no passengers are standing in practice, as seats may be occupied by luggage or passengers may choose to stand despite available seating.
Since the analysis focuses on passenger flows towards the city center, the highest capacity utilization occurs in the morning, reflecting the predominant commuting direction from suburban areas to the city center.
\begin{figure}[ht!]
	\centering
	\begin{minipage}{0.49\linewidth}
		\centering
		\includegraphics[width=1\linewidth]{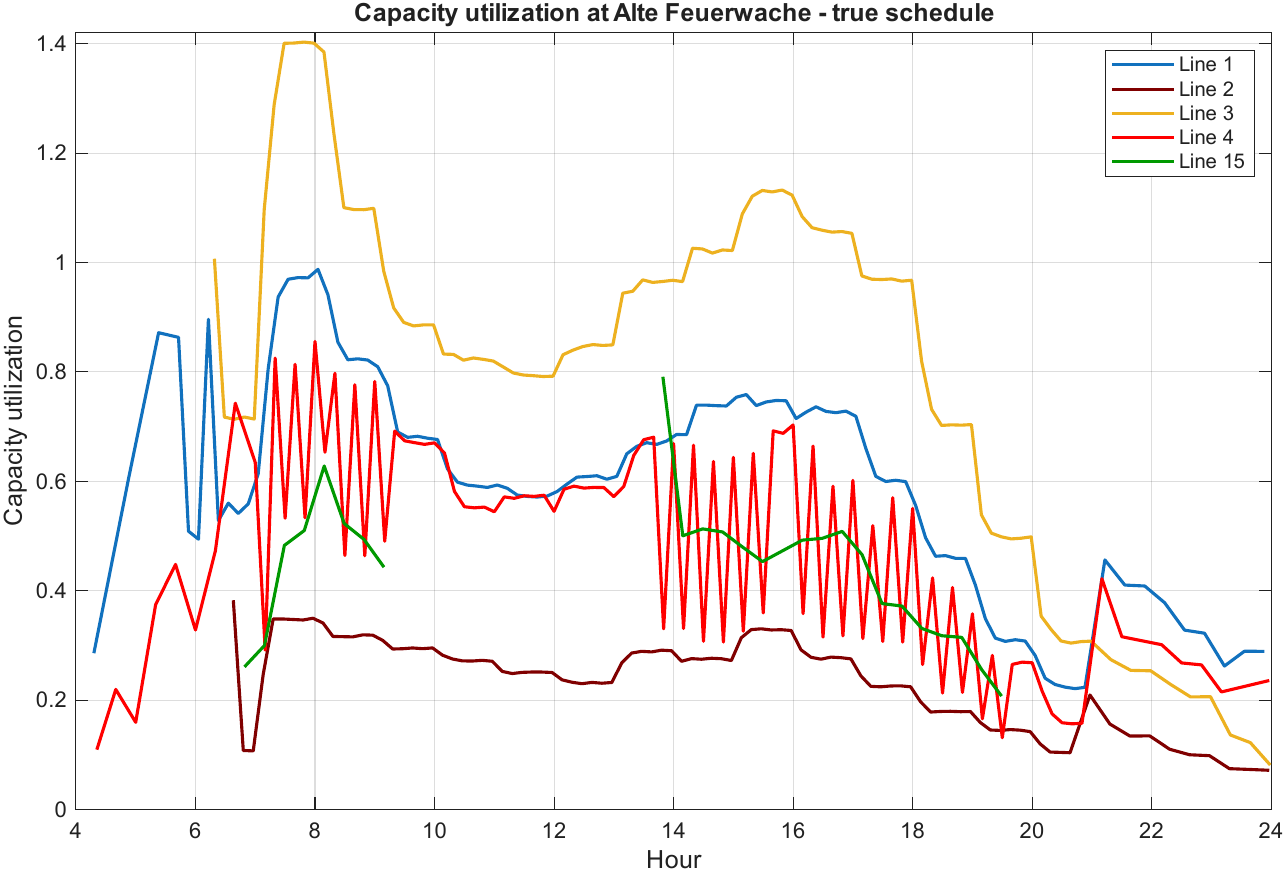}
		
	\end{minipage} ~\begin{minipage}{0.49\linewidth}
		\centering
		\includegraphics[width=1\linewidth]{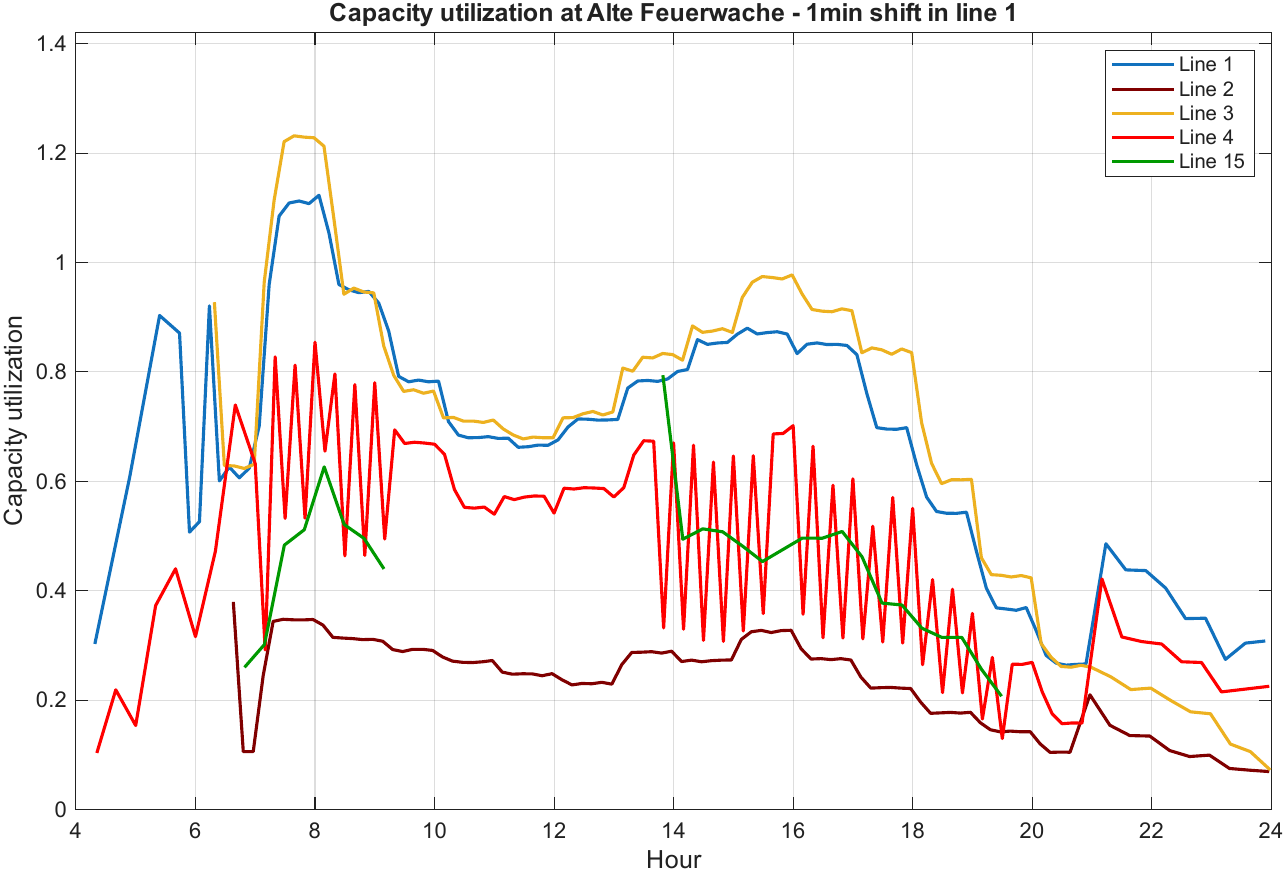}
		
	\end{minipage}
	\caption{The average capacity utilization of the five tram lines at Alte Feuerwache in time with the true schedule (left) and a shift the departure time of linie 1 by one minute.}
	\label{fig:CapacityUtilization}
\end{figure}
A further observable feature is a smaller peak around 9~p.m., coinciding with the transition to off-peak service frequencies. Although passenger demand is lower at this time, the reduced number of tram services leads to an increased capacity utilization.
Finally, an oscillatory pattern in the capacity utilization of line~4 can be observed. These oscillations occur exclusively during periods when line~15 is in operation. On the section between Bonifatiuskirche and Alte~Feuerwache, lines~4 and~15 share the same route. Whenever line~15 operates shortly before line~4, the utilization of line~4 is significantly reduced, as many waiting passengers board line~15 instead.

%
%\begin{figure}[ht!]
%	\centering
%	\includegraphics[width=0.75\linewidth]{Figures_Strassenbahn/network_utilization_normal.eps}
%	\caption{The average capacity utilization of the five tram lines at Alte Feuerwache in time with the true schedule.}
%	\label{fig:CapacityUtilization}
%\end{figure}
In order to better balance passenger allocation, particularly on the section between Luzenberg and Paradeplatz, we propose shifting the departure time of line~1 from \texttt{xx:x3} to \texttt{xx:x4}. This adjustment reduces the number of passengers traveling on line~3 and increases the capacity utilization of line~1, as passengers arriving within the time interval \([\texttt{xx:x3},\,\texttt{xx:x4})\) will now choose line~1 instead of line~3. The resulting capacity utilization is shown in the right panel of Figure~\ref{fig:CapacityUtilization}. In particular, this minor schedule modification prevents capacity utilization values exceeding~1 during the afternoon peak. Moreover, the maximum peak in the morning is significantly reduced. Both effects improve the passenger travel experience and reduce the risk of delays caused by large numbers of boarding and alighting passengers, as discussed in Section~\ref{sec:numerik_Delays}. 

Table \ref{tab:waitingTimesNetwork} shows the effects on the waiting and standing time in the entire network and at the tram stop Alte Feuerwache for different scenarios of shifting line 1 up to three minutes later. It is notable that a shift of one or two minutes later reduces both quantities significantly, while a shift of one minutes seems optimal on the discrete grid for the standing times.

%
%\begin{figure}[ht!]
%	\centering
%	\includegraphics[width=0.75\linewidth]{Figures_Strassenbahn/network_utilization_plus1.eps}
%	\caption{The average capacity utilization of the five tram lines at Alte Feuerwache in time with with a shift of line 1 to a departure time of xx:x4.}
%	\label{fig:CapacityUtilizationPlus1}
%\end{figure}

%Welche Analysen wollen wir fahren:
%\begin{itemize}
%	\item Auslastungszahlen der Linien nach Stunden, ggf. ueberlastung
%	\item Versuch der gleichmäßigeren Verteilung mit Verschiebung der Linie 1
%	\item Effekt auf Linie 4, wenn Linie 15 entlastet. (Bspw. Plot Queue Schafweide)
%\end{itemize}
%
%\textcolor{red}{Eine Tabelle zu den Warte und Stehzeiten bei unseren verschiedenen Verschiebungen}
\begin{table}[ht!]
	\centering
	\begin{tabular}{r||c|c|c|c}
		\hline
		&true schedule &  + 1min & + 2min &  + 3min\\
		\hline\hline
		total waiting time & 3498 & 3440 & 3458 & 3555\\\hline
		waiting time Alte Feuerwache & 248.3 & 234.6 &234.8 & 250.1\\\hline\hline
		total standing time & 11.07 & 4.35 & 4.63 & 11.93\\\hline
		standing time Alte Feuerwache & 2.08 & 1.09 & 1.37 & 2.85\\
	\end{tabular}
	\caption{Average waiting and standing times in hours for scenarios when shifting line 1 zero (true schedule), one, two or three minutes later.}
	\label{tab:waitingTimesNetwork}
\end{table}

	\section{Conclusion}
%	\textcolor{red}{TODO!!!!!}
%	This work introduces a PDE-based model for passenger flow in tram networks. Solutions are formulated in a measure-valued framework and governed by a hyperbolic transport equation. The applicability of the model is demonstrated in the numerical section, where classical traffic phenomena are reproduced in a case study based on the city of Mannheim. Initial effects of disruptions, such as delays and cancellations, are illustrated. A more detailed investigation of these effects, possibly embedded in a multi-objective optimization framework, would enable a systematic analysis of the trade-offs between societal objectives, such as accessibility and service robustness, and the economic and operational goals of the transport operator.
%	
%	The current model does not account for interactions between different lines. In particular, passenger transfers between lines are not explicitly modeled. Furthermore, passenger decision behavior is not incorporated. The model does not distinguish between passengers who board the next arriving tram and those who deliberately wait for a specific service. A rigorous modeling of such behaviors would require detailed trajectory data, which are currently available only to a limited extent and are therefore left for future research.

This work introduces a PDE-based model for passenger flow in tram networks, formulated in a measure-valued framework and governed by a hyperbolic transport equation. Its applicability is demonstrated numerically through a case study based on the city of Mannheim, where classical traffic phenomena as well as initial disruption effects are reproduced. A more detailed analysis of such disruptions within a multi-objective optimization framework could further clarify the trade-offs between societal objectives and the operational goals of the transport operator. The current model does not account for inter-line interactions or passenger decision behavior, such as transfers or waiting strategies for specific services. A rigorous treatment of these aspects would require detailed trajectory data, which are currently scarce and thus left for future research.

The proposed approach provides a promising foundation for the development of multi-modal traffic models. In particular, coupling established vehicular traffic models with the tram model on shared infrastructure segments represents a relevant direction for further research. Extending the framework to incorporate pedestrian and bicycle traffic may ultimately lead to an integrated and comprehensive transportation framework.
	
	\appendix
	\section{Proof of Theorem \ref{thm:UniqueSolutionNetwork}}
		\label{app: appendix}
In this appendix, we provide the proof of the existence theorem for a measure-valued solution on a single edge of the network.
	\begin{proof}[Proof of Theorem \ref{thm:UniqueSolutionNetwork}]
	We plug in the proposed measure-valued solution given by Theorem \ref{thm:UniqueSolutionOneEdge} into the left-hand side of Definition \ref{def: weakMeasureValuedSolutionOneEdge}:
	\begin{align*}
		&-\int_0^\infty \int_0^{l_e}  \varphi^e_t(x,t) + w^e\varphi^e_x(x,t)d\nu^e(x,t)\\
		= & -\int_0^\infty \int_0^{l_e}  
		\Big(\int_0^{\max \lbrace 0, (\theta^e)^{-1}(t)   \rbrace} \varepsilon_{\Phi_t^e(y,0)}(\lbrace x \rbrace) d\nu^e_{t=0}(y) + \int_{\max \lbrace 0, (\sigma^e)^{-1}(t) \rbrace}^t \varepsilon_{\Phi_t^e(0,s)}(\lbrace x \rbrace)d\tilde{\nu}^e_{x=0}(s) \\
		&~~+ \int_{\max \lbrace 0, (\sigma^e)^{-1}(t) \rbrace}^t\varepsilon_{\Phi_t^e(0,s)}(\lbrace x \rbrace)db^e(s) \Big)
		\varphi^e_t(x,t) \\
		&~~+ w^e\Big(\int_0^{\max \lbrace 0, (\theta^e)^{-1}(t)   \rbrace} \varepsilon_{\Phi_t^e(y,0)}(\lbrace x \rbrace) d\nu^e_{t=0}(y) + \int_{\max \lbrace 0, (\sigma^e)^{-1}(t) \rbrace}^t \varepsilon_{\Phi_t^e(0,s)}(\lbrace x \rbrace)d\tilde{\nu}^e_{x=0}(s) \\
		&~~+ \int_{\max \lbrace 0, (\sigma^e)^{-1}(t) \rbrace}^t\varepsilon_{\Phi_t^e(0,s)}(\lbrace x \rbrace)db^e(s)\Big)\varphi^e_x(x,t)dxdt.
	\end{align*}
	We apply Fubini-Tonellis theorem to interchange the integration of the two inner integrals, to evaluate the Dirac measures.
	\begin{align*}
		{=} & - \int_0^\infty \int_0^{\max \lbrace 0, (\theta^e)^{-1}(t)   \rbrace} \int_0^{l_e} \varepsilon_{\Phi_t^e(y,0)}(\lbrace x \rbrace) \varphi^e_t(x,t) +  w^e\varepsilon_{\Phi_t^e(y,0)}(\lbrace x \rbrace) \varphi^e_x(x,t)dx d\nu^e_{t=0}(y)dt\\
		& - \int_0^\infty \int_{\max \lbrace 0, (\sigma^e)^{-1}(t) \rbrace}^t \int_0^{l_e}\varepsilon_{\Phi_t^e(0,s)}(\lbrace x \rbrace) \varphi^e_t(x,t) + w^e\varepsilon_{\Phi_t^e(0,s)} (\lbrace x \rbrace)\varphi^e_x(x,t) dx d\tilde{\nu}^e_{x=0}(s) dt\\
		&- \int_0^\infty \int_{\max \lbrace 0, (\sigma^e)^{-1}(t) \rbrace}^t \int_0^{l_e}\varepsilon_{\Phi_t^e(0,s)}(\lbrace x \rbrace) \varphi^e_t(x,t) + w^e\varepsilon_{\Phi_t^e(0,s)} (\lbrace x \rbrace) \varphi^e_x(x,t) dx db^e(s) dt \phantom{+ \int_0^\infty  \int_0^{l_e}\varepsilon_{l_e}(\lbrace x \rbrace) \varphi^e_t(x,t) a^e(t) dx dt} \\
		%	&~~ - \int_0^\infty \int_0^{\max \lbrace 0, (\theta^e)^{-1}(t)   \rbrace} \int_0^{l_e}  w^e(x)\varepsilon_{\Phi_t^e(y,0)}(\lbrace x \rbrace) \varphi^e_x(x,t) dx d\nu^e_{t=0}(y)dt \\
		%	&~~- \int_0^\infty \int_{\max \lbrace 0, (\sigma^e)^{-1}(t) \rbrace}^{t} \int_0^{l_e}  w^e(x)\varepsilon_{\Phi_t^e(0,s)} (\lbrace x \rbrace)\varphi^e_x(x,t)dx d\tilde{\nu}^e_{x=0}(s)dt\\
		%	&~~- \int_0^\infty \int_{\max \lbrace 0, (\sigma^e)^{-1}(t) \rbrace}^t \int_0^{l_e} w^e(x)\varepsilon_{\Phi_t^e(0,s)} (\lbrace x \rbrace) \varphi^e_x(x,t)dx db^e(s)dt  \phantom{+\int_0^\infty  \int_0^{l_e}w^e(x)\varepsilon_{l_e} ( \lbrace x \rbrace) \varphi^e_x(x,t) a^e(t)dx dt}\\
		= & - \int_0^\infty \int_0^{\max \lbrace 0, (\theta^e)^{-1}(t)   \rbrace}   \varphi^e_t(\Phi_t^e(y,0),t) + w^e\varphi^e_x(\Phi_t^e(y,0),t) d\nu^e_{t=0}(y)dt\\
		& - \int_0^\infty \int_{\max \lbrace 0, (\sigma^e)^{-1}(t) \rbrace}^t  \varphi^e_t(\Phi_t^e(0,s),t) + w^e \varphi^e_x(\Phi_t^e(0,s),t) d\tilde{\nu}^e_{x=0}(s) dt\\
		&- \int_0^\infty \int_{\max \lbrace 0, (\sigma^e)^{-1}(t) \rbrace}^t  \varphi^e_t(\Phi_t^e(0,s),t) + w^e\varphi^e_x(\Phi_t^e(0,s),t) db^e(s) dt.
	\end{align*}
	Making use of the fact that with the definition of the push-forward in Equation \eqref{eq:PushForward} its time derivative is given by $\frac{d}{dt} \Phi^e_t(y,0) = w^e$, we can write the integrands as total time derivatives
	\begin{align*}
		= & - \int_0^\infty \int_0^{\max \lbrace 0, (\theta^e)^{-1}(t)   \rbrace} \frac{d}{dt} \varphi^e(\Phi_t^e(y,0),t)d\nu^e_{t=0}(y) dt\\
		& - \int_0^\infty \int_{\max \lbrace 0, (\sigma^e)^{-1}(t) \rbrace}^t \frac{d}{dt} \varphi^e(\Phi_t^e(0,s),t)d{\tilde{\nu}}^e_{x=0}(s) dt\\
		& - \int_0^\infty \int_{\max \lbrace 0, (\sigma^e)^{-1}(t) \rbrace}^t \frac{d}{dt}\varphi^e(\Phi_t^e(0,s),t)db^e(s)dt \phantom{+ \int_0^\infty  \varphi^e_t(l_e,t) + w^e(l_e) \varphi^e_x(l_e,t) da^e(t)}.
	\end{align*}
	Again applying the Theorem of Fubini-Tonelli and the fundamental theorem of calculus, to solve the integral with the total time derivative, yields:
	\begin{align*}
		= & -  \int_0^{l_e} \int_0^{\theta^e(y)} \frac{d}{dt} \varphi^e(\Phi_t^e(y,0),t)dt d\nu^e_{t=0}(y)  - \int_0^\infty \int_s^{\sigma^e(s)} \frac{d}{dt} \varphi^e(\Phi_t^e(0,s),t)dtd{\tilde{\nu}}^e_{x=0}(s) \\
		& - \int_0^\infty \int_s^{\sigma^e(s)} \frac{d}{dt}\varphi^e(\Phi_t^e(0,s),t)dtdb^e(s) \phantom{+ \int_0^\infty  \varphi^e_t(l_e,t) + w^e(l_e) \varphi^e_x(l_e,t) da^e(t)}\\
		%			=& -\int_0^{l_e}  [\varphi^e(\Phi_t^e(y,0),t)]_{t=0}^{t=\theta^e(y)} d\nu^e_{t=0}(y)  - \int_0^\infty  [\varphi^e(\Phi_t^e(0,s),t)]_{t=s}^{t={\sigma^e(s)}} d{\tilde{\nu}}^e_{x=0}(s) \\
		%			& - \int_0^\infty [\varphi^e(\Phi_t^e(0,s),t)]_{t=s}^{t={\sigma^e(s)}}db^e(s) + \int_0^\infty  \varphi^e_t(l_e,t) \phantom{+w^e(l_e) \varphi^e_x(l_e,t) da^e(t)}\\
		=& - \int_0^{l_e}{\varphi^e(l_e,\theta^e(y))d\nu^e_{t=0}(y)} - \varphi^e(\Phi_0^e(y,0),0)d\nu^e_{t=0}(y)\\
		& -\int_0^\infty {\varphi^e(l_e,\sigma^e(s)) d\tilde{\nu}_{x=0}^e(s)} - \varphi^e(\Phi_s^e(0,s),s)d\tilde{\nu}^e_{x=0}(s) \\
		& -\int_0^\infty {\varphi^e(l_e,\sigma^e(s)) db^e(s)}- \varphi^e(\Phi_s^e(0,s),s)db^e(s)\\
		%	&~~+ hiermussnochangepasst werden {\color{blue} +\int_0^\infty \varphi^e(l_e,\sigma^e(s)) da^e(s)} - \int_0^\infty \varphi^e(l_e,s)da^e(s)\\
		=& \int_0^{l_e}\varphi^e(y,0)d\nu^e_{t=0}(y)+ \int_0^\infty\varphi^e(0,s)d\tilde{\nu}^e_{x=0}(s) + \int_0^\infty \varphi^e(0,s)db^e(s) - \int_0^\infty\varphi^e(l_e,s)da^e(s) \\
		& - {\int_0^\infty \varphi^e(l_e,s)d\tilde{\nu}_{x=l_e}^e(s)},
	\end{align*}
	where we used that
	\begin{align*}
		\int_0^\infty \varphi^e(l_e,\sigma^e(s))d\tilde{\nu}_{x=0}^e(s) = \int_{\sigma^e(0)}^\infty \varphi^e(l_e,s)d\tilde{\nu}_{x=l_e}^e(s)
	\end{align*}
	and a similar equation also for the boarding passengers. Furthermore
	\begin{align*}
		\int_0^{l_e}\varphi^e(l_e,\theta^e(y))d\nu^e_{t=0}(y) = \int_0^{\theta^e(0)}\varphi^e(l_e,s)d\tilde{\nu}^e_{x=l_e}(s)
	\end{align*}
	and using $\theta^e(0)=\sigma^e(0)$ we obtain the integrals for the outflow measure $\tilde{\nu}_{x=l_e}$ by the following contributions
	\begin{align*}
		&\int_0^\infty \varphi^e(l_e,s)d\tilde{\nu}_{x=l_e}^e(s) \\&= \int_0^\infty \varphi^e(l_e,\theta^e(y))d\nu^e_{t=0}(y) +   \varphi^e(l_e,\sigma^e(s)) d\tilde{\nu}_{x=0}^e(s) +   \varphi^e(l_e,\sigma^e(s)) db^e(s).
	\end{align*}
	Furthermore we added a zero term for the alighting passengers which are the fourth part of the outflowing passengers in $\tilde{\nu}^e_{x=l_e}$ and there have to be considered as an additional term. This is exactly the right-hand side of the definition of a measure-valued solution in Definition \ref{def: weakMeasureValuedSolutionOneEdge}.

\end{proof}
	
	\printbibliography

@article{Camili2017,
title = {Transport of measures on networks},
journal = {Networks and Heterogeneous Media},
volume = {12},
number = {2},
pages = {191-215},
year = {2017},
%issn = {1556-1801},
%doi = {10.3934/nhm.2017008},
%url = {https://www.aimsciences.org/article/id/b12d7c9e-0023-47ce-9cfd-3ee2adf3c281},
author = {F. Camilli and R. De Maio and A. Tosin},
keywords = {Network, transport equation, measure-valued solutions, distribution conditions}
}

@article{Lueckerath2012,
author = {Lückerath, D. and Ullrich, O. and Speckenmeyer, E.},
year = {2012},
month = {08},
pages = {61-68},
title = {Modeling Time Table based Tram Traffic},
volume = {22},
journal = {Simulation Notes Europe},
%doi = {10.11128/sne.22.tn.10121}
}

@article{Evers2015,
title = {Mild solutions to a measure-valued mass evolution problem with flux boundary conditions},
journal = {Journal of Differential Equations},
volume = {259},
number = {3},
pages = {1068-1097},
year = {2015},
%issn = {0022-0396},
%doi = {https://doi.org/10.1016/j.jde.2015.02.037},
%url = {https://www.sciencedirect.com/science/article/pii/S0022039615001096},
author = {J.H.M. Evers and S. C. Hille and A. Muntean},
keywords = {Measure-valued equations, Flux boundary condition, Mild solutions, Boundary layer asymptotics, Singular limit, Convergence rate}
}

@article{Evers2018,
author = {Evers, J. H. M. and Hille, S. C. and Muntean, A.},
title = {Measure-Valued Mass Evolution Problems with Flux Boundary Conditions and Solution-Dependent Velocities},
journal = {SIAM Journal on Mathematical Analysis},
volume = {48},
number = {3},
pages = {1929-1953},
year = {2016}
}

@article{Piccoli2014,
  author    = {B. Piccoli and F. Rossi},
  title     = {Generalized Wasserstein Distance and its Application to Transport Equations with Source},
  journal   = {Archive for Rational Mechanics and Analysis},
  year      = {2014},
  volume    = {211},
  number    = {1},
  pages     = {335--358},
%  doi       = {10.1007/s00205-013-0669-x},
%  url       = {https://doi.org/10.1007/s00205-013-0669-x},
%  issn      = {1432-0673}
}

@article{Camilli2018,
title = {Measure-valued solutions to nonlocal transport equations on networks},
journal = {Journal of Differential Equations},
volume = {264},
number = {12},
pages = {7213-7241},
year = {2018},
%issn = {0022-0396},
%doi = {https://doi.org/10.1016/j.jde.2018.02.015},
%url = {https://www.sciencedirect.com/science/article/pii/S0022039618300986},
author = {F. Camilli and R. {De Maio} and A. Tosin},
keywords = {Network, Transport equation, Measure-valued solution, Transmission conditions}
}

@book{Kazuaki2014,
series = {Springer Monographs in Mathematics},
%isbn = {9783662436967},
year = {2014},
title = {Semigroups, Boundary Value Problems and Markov Processes},
edition = {2nd ed. 2014},
%language = {eng},
address = {Berlin, Heidelberg s.l},
author = {Taira, K.},
keywords = {Mathematics; Differential equations Partial; Distribution (Probability theory); Probabilities; Harmonic analysis; Functional analysis; Boundary value problems; Markov processes; Semigroups},
}

@article{Dorini.2011,
 author = {Dorini, F. A. and Cunha, M. C.C.},
 year = {2011},
 title = {On the linear advection equation subject to random velocity fields},
% url = {https://EconPapers.repec.org/RePEc:eee:matcom:v:82:y:2011:i:4:p:679-690},
 keywords = {35R60;60H15;Advection;Gaussian processes;Random;Velocity},
 pages = {679--690},
 volume = {82},
 number = {4},
 journal = {Mathematics and Computers in Simulation}
}

@article{Gugat.2015f,
 author = {Gugat, M. and Keimer, A. and Leugering, G. and Wang, Z.},
 year = {2015},
 title = {Analysis of a system of nonlocal conservation laws for multi-commodity flow on networks},
 %url = {https://www.aimsciences.org/article/id/8d86946c-ce0a-49d5-b915-ce76bf28019d},
% keywords = {Conservation laws;conservation laws on networks;Existence of minimizers;multi-commodity model;nonlocal %conservation laws;optimal control of conservation laws on networks;optimal nodal control;systems of hyperbolic PDEs},
 pages = {749--785},
 volume = {10},
 number = {4},
 journal = {Networks and Heterogeneous Media},
% doi = {10.3934/nhm.2015.10.749}
}

@article{DiPerna1989,
author = {DiPerna, R.J. and Lions, P.L.},
journal = {Inventiones mathematicae},
keywords = {global solutions; transport equation; kinetic Vlasov-type models; fluid mechanics},
number = {3},
pages = {511-548},
title = {Ordinary differential equations, transport theory and Sobolev spaces.},
%url = {http://eudml.org/doc/143741},
volume = {98},
year = {1989}
}

@article{Gottlich.2022b,
 author = {Göttlich, S. and Schillinger, T.},
 year = {2022},
 title = {Control strategies for transport networks under demand uncertainty},
 pages = {74},
 volume = {48},
 number = {6},
 %issn = {1572-9044},
 journal = {Advances in Computational Mathematics}
 %doi = {10.1007/s10444-022-09993-9}
}

@article{Gottlich.2019b,
 author = {G{\"o}ttlich, S. and Korn, R. and Lux, K.},
 year = {2019},
 title = {Optimal control of electricity input given an uncertain demand},
 pages = {1--28},
 volume = {90},
 journal = {Mathematical Methods of Operations Research}
% doi = {10.1007/s00186-019-00678-6}
}

@article{LiRao2004,
  author  = {Li, T. and Rao, B.},
  title   = {Exact Boundary Controllability of Unsteady Flows in a Tree-like Network of Open Canals},
  journal = {Methods and Applications of Analysis},
  volume  = {11},
  number  = {3},
  pages   = {353--366},
  year    = {2004},
  month   = sep
}

@book{LEVEQUE.2002,
 author = {LeVeque, R. J.},
 year = {2002},
 title = {Finite Volume Methods for Hyperbolic Problems},
 publisher = {{Cambridge University Press}},
 series = {Cambridge Texts in Applied Mathematics},
 volume = {31},
 address = {Cambridge},
% doi = {10.1017/CBO9780511791253}
}

@InProceedings{bortoletto2024,
  author =	{Bortoletto, E. and van Lieshout, R. N. and Masing, B. and Lindner, N.},
  title =	{{Periodic Event Scheduling with Flexible Infrastructure Assignment}},
  booktitle =	{24th Symposium on Algorithmic Approaches for Transportation Modelling, Optimization, and Systems (ATMOS 2024)},
  pages =	{4:1--4:18},
%  series =	{Open Access Series in Informatics (OASIcs)},
%  ISBN =	{978-3-95977-350-8},
%  ISSN =	{2190-6807},
  year =	{2024},
  volume =	{123},
  editor =	{Bouman, Paul C. and Kontogiannis, Spyros C.},
  publisher =	{Schloss Dagstuhl -- Leibniz-Zentrum f{\"u}r Informatik},
%  address =	{Dagstuhl, Germany}
%  URL =		{https://drops.dagstuhl.de/entities/document/10.4230/OASIcs.ATMOS.2024.4},
%  URN =		{urn:nbn:de:0030-drops-211929},
%  doi =		{10.4230/OASIcs.ATMOS.2024.4},
%  annote =	{Keywords: Periodic Event Scheduling, Periodic Timetabling, Railway Timetabling, Matchings, Infrastructure %Assignments, Platform Assignments, Station Capacities}
}

@article{Serafini1989,
  title={A Mathematical Model for Periodic Scheduling Problems},
  author={P. Serafini and W. Ukovich},
  journal={SIAM J. Discret. Math.},
  year={1989},
  volume={2},
  pages={550-581}
%  url={https://api.semanticscholar.org/CorpusID:19409360}
}

@article{Masing2023,
title = {Periodic timetabling with integrated track choice for railway construction sites},
journal = {Journal of Rail Transport Planning \& Management},
volume = {28},
pages = {100416},
year = {2023},
%issn = {2210-9706},
%doi = {https://doi.org/10.1016/j.jrtpm.2023.100416},
%url = {https://www.sciencedirect.com/science/article/pii/S2210970623000483},
author = {B. Masing and N. Lindner and C. Liebchen},
keywords = {Railway timetabling, Periodic timetabling, Periodic event scheduling, Train routing, Turnarounds}
}

@book{Iwnicki.2019,
  editor    = {Iwnicki, S. and Spiryagin, M. and Cole, C. and McSweeney, T.},
  title     = {Handbook of Railway Vehicle Dynamics},
  edition   = {2},
  publisher = {CRC Press},
  year      = {2019},
 % doi       = {10.1201/9780429469398}
}

@article{Lighthill.1955,
 author = {Lighthill, M. J. and Whitham, G. B.},
 year = {1955},
 title = {On Kinematic Waves. II. A Theory of Traffic Flow on Long Crowded Roads},
 pages = {317--345},
 volume = {229},
 number = {1178},
 journal = {Proceedings of the Royal Society of London. Series A, Mathematical and Physical Sciences}
}

@article{RICHARDS.1956,
 author = {Richards, P. I.},
 year = {1956},
 title = {Shock Waves on the Highway},
 pages = {42--51},
 volume = {4},
 number = {1},
 journal = {Operations Research}
}

@book{GARAVELLO.2006,
 author = {Garavello, M. and Piccoli, B.},
 year = {2006},
 title = {Traffic flow on networks},
 address = {Springfield, Mo},
 volume = {1},
 publisher = {{AIMS Series on  Applied Mathematics, American Institute of Mathematical Sciences (AIMS)}},
 series = {AIMS series on applied mathematics}
}

@book{GARAVELLO.2016,
 author    = {Garavello, M. and Han, K. and Piccoli, B.},
 year      = {2016},
 title     = {Models for Vehicular Traffic on Networks},
 address   = {Springfield, Mo},
 volume    = {9},
 publisher = {{AIMS Series on Applied Mathematics, American Institute of Mathematical Sciences (AIMS)}},
 series    = {AIMS series on applied mathematics}
}

@article{BERTSCH.2024,
title = {Measure-valued solutions of scalar hyperbolic conservation laws, Part 1: Existence and time evolution of singular parts},
journal = {Nonlinear Analysis},
volume = {245},
%pages = {113571},
year = {2024},
%issn = {0362-546X},
%doi = {https://doi.org/10.1016/j.na.2024.113571},
%url = {https://www.sciencedirect.com/science/article/pii/S0362546X24000907},
author = {M. Bertsch and F. Smarrazzo and A. Terracina and A. Tesei},
keywords = {First order hyperbolic conservation laws, Radon measure-valued entropy solutions, Continuity properties, Compatibility conditions}
}

@article{Colombo.2012,
author = {Colombo, R. M. and Garavello, M. and L\'{e}cureux-Mercier, M.},
title = {A class of nonlocal models for pedestrian traffic},
journal = {Mathematical Models and Methods in Applied Sciences},
volume = {22},
number = {04},
%pages = {115003},
year = {2012}
%doi = {10.1142/S0218202511500230},
%URL = {
%        https://doi.org/10.1142/S0218202511500230
%},
%eprint = {
%        https://doi.org/10.1142/S0218202511500230 
%   }
}

@article{Joumaa.2026,
title = {A macroscopic framework for modeling heterogeneous traffic flows on urban networks},
journal = {Applied Mathematical Modelling},
volume = {156},
pages = {116789},
year = {2026},
%issn = {0307-904X},
%doi = {https://doi.org/10.1016/j.apm.2026.116789},
%url = {https://www.sciencedirect.com/science/article/pii/S0307904X26000508},
author = {A. Joumaa and P. Goatin and G. {De Nunzio}},
keywords = {Multi-class macroscopic traffic flow models on networks, Hyperbolic systems of conservation laws, Finite volume schemes, Traffic management}
}

@article{Gwiazda2012,
author = {Gwiazda, P. and Jamr\'{o}z, G. and Marciniak-Czochra, A.},
title = {Models of Discrete and Continuous Cell Differentiation in the Framework of Transport Equation},
journal = {SIAM Journal on Mathematical Analysis},
volume = {44},
number = {2},
pages = {1103-1133},
year = {2012}
}

@inproceedings{Roelofsen2018,
  author    = {Roelofsen, D. and Cats, O. and van Oort, N. and Hoogendoorn, S.},
  title     = {Assessing disruption management strategies in rail-bound urban public transport systems from a passenger perspective},
  booktitle = {Proceedings of the Conference on Advanced Systems in Public Transport (CASPT) 2018},
  year      = {2018},
  address   = {Brisbane, Australia},
  month     = jul,
  pages     = {Article 21}
}

@article{NARAYAN2017,
title = {Performance assessment of fixed and flexible public transport in a multi agent simulation framework},
journal = {Transportation Research Procedia},
volume = {27},
pages = {109-116},
year = {2017},
note = {20th EURO Working Group on Transportation Meeting, EWGT 2017, 4-6 September 2017, Budapest, Hungary},
author = {J. Narayan and O. Cats and N. {van Oort} and S. Hoogendoorn}
}

@article{Courant.1928,
  title={{\"U}ber die partiellen Differenzengleichungen der mathematischen Physik},
  author={R. Courant and K. Friedrichs and H. Lewy},
  journal={Mathematische Annalen},
  year={1928},
  volume={100},
  pages={32-74}
}

@article{Ogata.1981,
 author = {Ogata, Y.},
 year = {1981},
 title = {On Lewis' simulation method for point processes},
 pages = {23--31},
 volume = {27},
 number = {1},
 journal = {IEEE Transactions on Information Theory}
}

@article{Herty2003,
author = {Herty, M. and Klar, A.},
title = {Modeling, Simulation, and Optimization of Traffic Flow Networks},
journal = {SIAM Journal on Scientific Computing},
volume = {25},
number = {3},
pages = {1066-1087},
year = {2003}
}
	
\end{document}